\documentclass[10pt]{article}
\usepackage{amsthm,amsfonts,amsmath,amscd,amssymb,verbatim}
\usepackage{graphicx}
\usepackage{psfrag}

\usepackage{pinlabel} 
\title{Hamiltonian unknottedness of certain monotone Lagrangian tori in $S^2\times S^2$}
\author{K.~Cieliebak and M.~Schwingenheuer}
\date{}
\theoremstyle{plain}
\newtheorem{theorem}{Theorem}[section]
\newtheorem{thm}[theorem]{Theorem}

\newtheorem{cor}[theorem]{Corollary}
\newtheorem{proposition}[theorem]{Proposition}
\newtheorem{prop}[theorem]{Proposition}
\newtheorem{lemma}[theorem]{Lemma}

\theoremstyle{remark}

\newtheorem{remark}[theorem]{Remark}

\newtheorem{definition}[theorem]{Definition}

\newcommand{\id}{\mathrm{id}}
\newcommand{\Id}{{{\mathchoice {\rm 1\mskip-4mu l} {\rm 1\mskip-4mu l}
{\rm 1\mskip-4.5mu l} {\rm 1\mskip-5mu l}}}}

\newcommand{\ol}{\overline}

\newcommand{\p}{\partial}

\newcommand{\om}{\omega}
\newcommand{\Om}{\Omega}
\newcommand{\eps}{\varepsilon}
\newcommand{\into}{\hookrightarrow}
\newcommand{\la}{\langle}
\newcommand{\ra}{\rangle}
\newcommand{\wt}{\widetilde}
\newcommand{\wh}{\widehat}

\newcommand{\Z}{{\mathbb{Z}}}
\newcommand{\R}{{\mathbb{R}}}
\newcommand{\C}{{\mathbb{C}}}

\renewcommand{\max}{{\rm max}}

\newcommand{\inn}{{\rm int\,}}

\newcommand{\pt}{{\rm pt}}

\newcommand{\Diff}{{\rm Diff}}
\newcommand{\Symp}{{\rm Symp}}
\newcommand{\Ham}{{\rm Ham}}
\newcommand{\std}{{\rm std}}
\newcommand{\lh}{{\rm lh}}

\newcommand{\FF}{\mathcal{F}}

\newcommand{\HH}{\mathcal{H}}

\renewcommand{\SS}{\mathfrak{S}}
\newcommand{\dd}[2]{\frac{\partial {#1}}{\partial {#2}}}

\newcommand{\fo}[1]{\mathcal{{#1}}}
\newcommand{\tv}[0]{\tilde{v}}

\newcommand{\ostd}[0]{\sigma_\std}   

\newcommand{\lstd}[0]{\lambda_{std}}
\begin{document}

\maketitle

\begin{abstract}
\noindent
We prove that a monotone Lagrangian torus in $S^2\times S^2$ which
suitably sits in a symplectic fibration with two sections in its
complement is Hamiltonian isotopic to the Clifford torus. 
\end{abstract}

\section{Introduction}\label{intro}

The classification of Lagrangian submanifolds in symplectic manifolds
up to Lagrangian or Hamiltonian isotopy is an intriguing 
problem of symplectic topology. While there are many tools for
distinguishing Lagrangian submanifolds, actual classification results 
have been very rare and restricted to special manifolds in dimension
$4$. 
The first circle of results concerns Lagrangian $2$-spheres, in which
case the two notions of isotopy coincide: There is a unique
$2$-sphere up to Hamiltonian isotopy in $S^2\times S^2$
(Hind~\cite{Hind04}), in $T^*S^2$ and some other Stein 
surfaces (Hind~\cite{Hind10}), and in certain del Pezzo surfaces 
(Evans~\cite{jonny}). 
The second circle of results is contained in A.~Ivrii's PhD
thesis~\cite{Ivrii}: It asserts the uniqueness up to Lagrangian isotopy
of Lagrangian tori in $\mathbb{R}^4$, $S^2\times S^2$,
$\mathbb{CP}^2$, and $T^\ast \mathbb{T}^2$.
 
Motivated by Ivrii's thesis, we address in this paper the question of
{\em Hamiltonian} unknottedness of monotone Lagrangian tori in
$S^2\times S^2$. 
Recall that a Lagrangian torus is called {\em monotone} if its Maslov
class is a positive multiple of its symplectic area class on relative
$\pi_2$. The product of the equators in each $S^2$-factor in
$S^2\times S^2$ is called the {\em standard Lagrangian torus}
$L_\std$, or the {\em Clifford torus}. This torus is monotone for the
standard split symplectic form $\om_\std=\ostd\oplus \ostd$.
There have been many constructions of monotone Lagrangian tori in
$(S^2\times S^2,\om_\std)$ that are not Hamiltoniian
isotopic to $L_\std$ due to Eliashberg--Polterovich~\cite{ElP},
Chekanov--Schlenk~\cite{CS}, 
Entov--Polterovich~\cite{EnP}, Biran--Cornea~\cite{BC},
Fukaya--Oh--Ohta--Ono~\cite{FOOO}, and Albers--Frauenfelder~\cite{AF}. 
Since they were all recently shown to be Hamiltonian isotopic to each 
other~\cite{Gad,OU}, we will collectively refer to them as the {\em
  Chekanov-Schlenk torus} $L_{CS}$. 

The following definition is implicit in Ivrii's thesis. Let us call a 
monotone Lagrangian torus $L$ in $(S^2\times S^2,\om_\std)$
{\em fibered} if there exists a foliation $\fo{F}$ of $S^2\times
S^2$ by symplectic $2$-spheres in the homology class $[pt\times S^2]$
and a symplectic submanifold $\Sigma$ in class $[S^2\times pt]$ with
the following properties: 
\begin{itemize}
 \item $\Sigma$ is transverse to the leaves of $\fo{F}$ and is
   disjoint from $L$; 
\item the leaves of $\fo{F}$ intersect $L$ in circles (or not at all).
\end{itemize}
Note that each leaf of $\fo{F}$ which intersects the torus $L$ is cut
by $L$ into two closed disks glued along $L$. The disks that intersect
$\Sigma$ together form a solid torus $T$ with $\partial T=L$. 

The first part of Ivrii's thesis now asserts that {\em any
  monotone Lagrangian torus in $S^2\times S^2$ is fibered}.  
In this paper, we prove\footnote{
Let us emphasize that, while inspired by it, the results in this paper
are independent of Ivrii's thesis (which is, unfortunately, neither
published nor otherwise accessible).} 

\begin{thm}[Main Theorem]\label{thm:main}
Let $L\subset (S^2\times S^2,\om_\std)$ be a monotone
Lagrangian torus which is fibered by $\fo{F}$ and $\Sigma$. 
Assume in addition, that there exists a second symplectic submanifold
$\Sigma'$ in homology class $[S^2\times pt]$ which is transverse to
the leaves of $\fo{F}$, and which is disjoint from $\Sigma$ and
$T$. Then $L$ is Hamiltonian isotopic to the standard torus $L_\std$. 
\end{thm}

This means that the Chekanov--Schlenk torus $L_{CS}$, or any other
exotic monotone Lagrangian torus, cannot possess the additional
section $\Sigma'$ required in Theorem~\ref{thm:main}. It also suggests 
that the classification of monotone Lagrangian tori in $(S^2\times
S^2,\om_\std)$ up to Hamiltonian isotopy may come within
reach once we understand better the role of the second section
$\Sigma'$. 

\begin{remark}
It is tempting to conjecture that the Clifford torus and the
Chekanov--Schlenk torus are the only monotone Lagrangian tori in
$(S^2\times S^2,\om_\std)$ up to Hamiltonian isotopy. We
point out that R.~Vianna~\cite{Via1,Via2} recently constructed
infinitely many pairwise Hamiltonian non-isotopic Lagrangian tori in
$\C P^2$, and conjectured the existence of a monotone Lagrangian tori
in $(S^2\times S^2,\om_\std)$ which is not Hamiltonian
isotopic to either $L_\std$ or $L_{CS}$. 
\end{remark}

Let us now outline the proof of the main theorem, and in particular
explain where the second section is needed. By a \emph{relative
  symplectic fibration} on $S^2\times S^2$ we will mean a quintuple 
$$
   \SS=(\fo{F},\omega,L,\Sigma,\Sigma'),
$$ 
as in Theorem~\ref{thm:main}, only with the standard form $\om_\std$
replaced by any symplectic form $\om$ cohomologous to $\om_\std$. 
We will prove (Corollary~\ref{cor:fib}) that for every symplectic
fibration $\SS$ with $\om=\om_\std$ there exists a homotopy of
relative symplectic fibrations
$\SS_t=(\fo{F}_t,\om_\std,L_t,\Sigma_t,\Sigma_t')$ with 
fixed symplectic form $\om_\std$ such that 
$\SS_0=\SS$ and $\SS_1=\SS_{std}:=(\fo{F}_\std,\omega_\std,L_\std,S_0,S_\infty)$,
where $\FF_\std$ denotes the standard foliation with leaves
$\{z\}\times S^2$ and $S_0=S^2\times\{S\}$, $S_\infty=S^2\times\{N\}$
are the standard sections at the south and north pole. Then $L_t$ is
an isotopy of monotone Lagrangian tori with respect to $\om_\std$ from
$L$ to $L_\std$, which is Hamiltonian by Banyaga's isotopy extension
theorem.  

A relative symplectic fibration $\SS$ gives rise to a symplectic
fibration $p:S^2\times S^2\to\Sigma$ by sending each leaf of $\FF$ to
its intersection point with $\Sigma$. It determines a symplectic
connection whose horizontal subspaces are the symplectic orthogonal
complements to the fibres. Parallel transport along closed paths
$\gamma:[0,1]\to\Sigma$ gives holonomy maps which are
symplectomorphisms of the fibre $p^{-1}(\gamma(0))$ and measure the
non-integrability of the horizontal distribution. It is not hard to
show that a symplectic fibration $\SS$ with trivial holonomy around
all loops is diffeomorphic to $\SS_\std$, and a theorem of Gromov
implies that they are actually homotopic with fixed symplectic form if
they both have symplectic form $\om_\std$. 

Thus, most of the work will go into deforming a given relative
symplectic fibration $\SS$ to one with trivial holonomy. 
After pulling back $\SS$ by a diffeomorphism, we may assume that
$(\FF,L,\Sigma,\Sigma')=(\FF_\std,L_\std,S_0,S_\infty)$ (but the
symplectic form $\om$ is nonstandard). 
In the first step, which takes up Section~\ref{sec:stand},
we make the holonomy trivial near the two sections and near the fibres
over the line of longitude through Greenwich in the base, see Figure~\ref{fi1}. 

\begin{figure}[h!]
 \centering
 \psfrag{G}{Greenwich}
\psfrag{S}{$\Sigma'$}
\psfrag{Sd}{$\Sigma$}
\psfrag{L}{$L$}
\psfrag{bases}{Fibre}
\psfrag{base}{Base}
 \includegraphics[width=10cm]{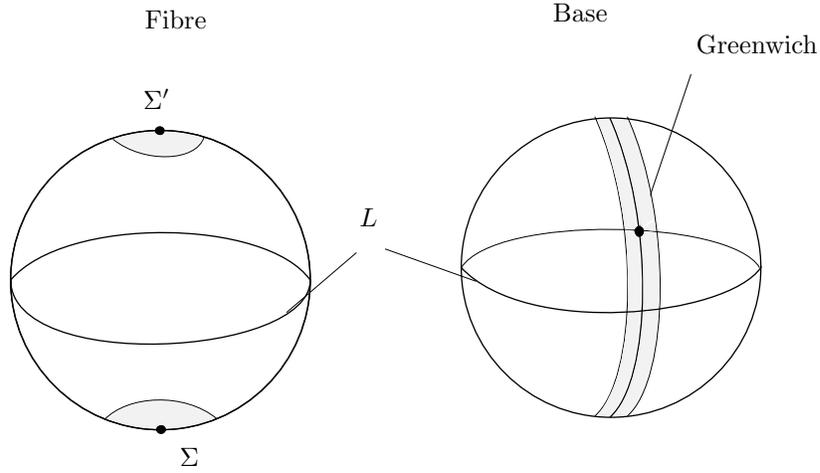}
 \caption{Where the holonomy is trivial after the first step}
 \label{fi1}
\end{figure}

In the second step, which takes up most of Section~\ref{sec:killholonomy},
we kill the holonomy along all circles of latitude $C^\lambda$. 
For this, let $(\lambda,\mu)$ be spherical coordinates on $S^2$,
where $\lambda$ denotes the latitude and $\mu$ denotes the longitude. 
After the first step, the holonomy maps $\phi^\lambda$ along $C^\lambda$ give
a loop in $\Symp(A,\partial A,\ostd)$, the group of symplectomorphisms
of the annulus (the sphere minus two polar caps) which equal the
identity near the boundary. Since the fundamental group of
$\Symp(A,\partial A,\ostd)$ vanishes, we can contract the loop of inverses
$\psi^\lambda=(\phi^\lambda)^{-1}$ and obtain a family of Hamiltonians
$H^\lambda_\mu$ which generates the contraction. The closed $2$-form 
$$
   \Omega_H=\omega+d(H^\lambda_\mu d\mu)
$$
then defines a symplectic connection with trivial holonomy around all 
$C^\lambda$. However, $\Omega_H$ need not be symplectic
if $\dd{H^\lambda_\mu}{\lambda}$ is large. This can be remedied by the
inflation procedure due to Lalonde and McDuff~\cite{LalMcD}. 
In this procedure, the symplectic form $\om$ is deformed along a fibre
and a section (and $H$ suitably rescaled) in order to make $\Om_H$
symplectic. However, this process will in general destroy monotonicity
of $L_\std$. In order to keep the Lagrangian torus monotone, we
perform the inflation procedure along a fibre and the {\em two}
sections $S_0,S_\infty$ in a symmetric way. It is at this point of the
proof that we need the existence of a second symplectic section.

Once the holonomy along circles of latitude is trivial, in the third
and final step (at the end of Section~\ref{sec:killholonomy}) we
deform the symplectic form to the standard form. 
This finishes the outline of the proof.

\begin{remark}
The idea to apply the results of Ivrii's thesis to the Hamiltonian
classification of monotone tori in $S^2\times S^2$ originated in 2003
in the first author's discussions with Y.~Eliashberg. However, at the
time we did not realize the necessity of a second symplectic section
and were puzzled by the apparent contradiction between this result and
the existence of an exotic monotone torus in $S^2\times S^2$. 
This discrepancy was resolved in the second author's PhD
thesis~\cite{Sch}, of which this article is a shortened version. 
\end{remark}

\bigskip

\centerline{\bf Acknowledgement}

We thank Y.~Eliashberg for many fruitful discussions and his
continued interest in this work.

\section{Relative symplectic fibrations}\label{fib}

\subsection{Symplectic connections and their holonomy}
Consider a smooth fibration (by which we mean a fibre bundle) $p:M\to
B$ and a closed $2$-form $\om$  
on $M$ whose restriction to each fibre $p^{-1}(b)$ is
nondegenerate. We will refer to $\om$ as a {\em symplectic connection}
on $M$.\footnote{This terminology differs slightly from the one in~\cite{MS}.} 
From the next subsection on we will assume $\om$ to be symplectic, but
for now this is not needed. 

{\bf Parallel transport. }
Since $\om$ is nondegenerate on the fibres, the $\om$-orthogonal
complements  
$$
   \HH_x := (\ker d_xp)^\om
$$
to the tangent spaces of the fibres of $p$ define a distribution of
horizontal subspaces $\HH$ such that 
$$
   TM = \HH\oplus \ker dp.  
$$
Horizontal lifts of a path $\gamma:[0,1]\to B$ with given initial
points in $p^{-1}(\gamma(0))$ give rise to the {\em parallel transport}
$$
   P_\gamma\colon p^{-1}(\gamma(0)) \to p^{-1}(\gamma(1))
$$ 
along $\gamma$. Closedness of $\om$ inplies that $P_\gamma$ is
symplectic, i.e.
$$
   P_\gamma^\ast \omega_{\gamma(1)}=\omega_{\gamma(0)},
$$ 
where $\omega_b$ denotes the symplectic form $\omega|_{p^{-1}(b)}$.

{\bf Holonomy. }
The parallel transport $P_\gamma:p^{-1}(\gamma(0))\to p^{-1}(\gamma(0))$
along a closed curve $\gamma:[0,1]\to B$ is called the {\em holonomy}
of $\om$ along the loop $\gamma$. If $P_\gamma=\id$ for each loop
$\gamma$ we say that $\om$ has {\em trivial holonomy}. In this case
parallel transport along any (not necessarily closed) curve depends
only on the end points, so we can use parallel transport to define
local trivializations of $p:M\to B$.

\begin{remark}
There is a natural notion of {\em curvature} of a symplectic connection,
see~\cite{MS}. This is a $2$-form on the base with values in the
functions on the fibres which measures the nonintegrability of the
horizontal distribution. For simply connected base (which is the case
of interest to us) the curvature and the holonomy carry the same
information, so in this paper we will phrase everything in terms of
holonomy.  
\end{remark}

{\bf From foliations to fibrations. }
More generally, we can consider a closed $2$-form $\om$ on $M$ whose
restriction to the leaves of a smooth foliation $\FF$ of $M$ is
nondegenerate. If all leaves of $\FF$ are compact,
then the space of leaves is a smooth manifold $B$ and the canonical
projection $M\to B$ is a fibration, so we are back in the situation
of a symplectic connection as above. Since in our case all leaves will
be $2$-spheres, we can switch freely between the terminologies of
foliations and fibrations.

\subsection{Fibered Lagrangian tori in $S^2\times S^2$}

Suppose now that $(M,\om)$ is a symplectic $4$-manifold and $p:M\to B$
is a symplectic fibration over a surface $B$ (i.e., the fibres are
symplectic surfaces). 

\begin{definition}\label{DLfibered}
We say that an embedded $2$-torus $L\subset M$ is \emph{fibered by
  $p$} if (see Figure~\ref{f3})
\begin{enumerate}
\item $\gamma:=p(L)$ is an immersed loop with transverse
  self-intersections which are at most double points; 
\item $p^{-1}(\gamma(t))\cap L$ is diffeomorphic to a circle if
  $\gamma(t)$ is not a double point, and to two disjoint circles if
  $\gamma(t)$ is a double point;
\item in each of the circles in $p^{-1}(\gamma(t))\cap L$ we can fill
  in an embedded disk $D\subset p^{-1}(\gamma(t))$ in the fibre such
  that the two disks at a double point are disjoint and all the disks
  form a solid torus $T\cong S^1\times D^2$ with $L$ as its boundary. 
\end{enumerate}
\end{definition}

\begin{figure}[h!]
 \centering
\psfrag{L}{$L=\partial T\subset M$}
\psfrag{p}{$p$}
\psfrag{B}{$B$}
\psfrag{pcupT}{$p^{-1}(\gamma(t))\cap T$}
\psfrag{gamma}{$\gamma$}
\psfrag{x}{$\gamma(t)$}
\psfrag{a}{$p^{-1}(\gamma(t))$}
\includegraphics[width=10cm]{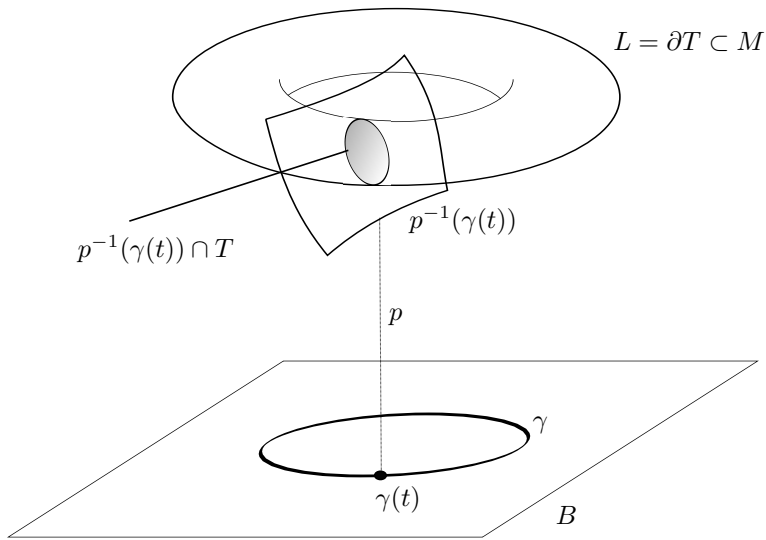}
\caption{A fibered Lagrangian torus}
\label{f3}
\end{figure}

Suppose now that $L$ is in addition Lagrangian. The following two
results are the basis for most of the sequel. The first one
states that a fibered Lagrangian torus $L$ is generated by parallel
transport along $\gamma$ of the circle in the fibre over a non-double
point, see Figure~\ref{f4}.

\begin{figure}[h!]
\centering
\psfrag{pi}{$p$}
\psfrag{g}{$\gamma$}
\psfrag{g1}{$\tilde{\gamma}_{z_1}$}
\psfrag{g2}{$\tilde{\gamma}_{z_2}$}
\psfrag{g3}{$\tilde{\gamma}_{z_3}$}
\psfrag{g4}{$\tilde{\gamma}_{z_4}$}
\psfrag{F1}{$p^{-1}(\gamma(t_0))$}
\psfrag{F2}{$p^{-1}(\gamma(t_1))$}
\psfrag{a}{$p^{-1}(\gamma(t_0))\cap L$}
\psfrag{b}{$p^{-1}(\gamma(t_1))\cap L$}
\psfrag{B}{$B$}
 \includegraphics[width=10cm]{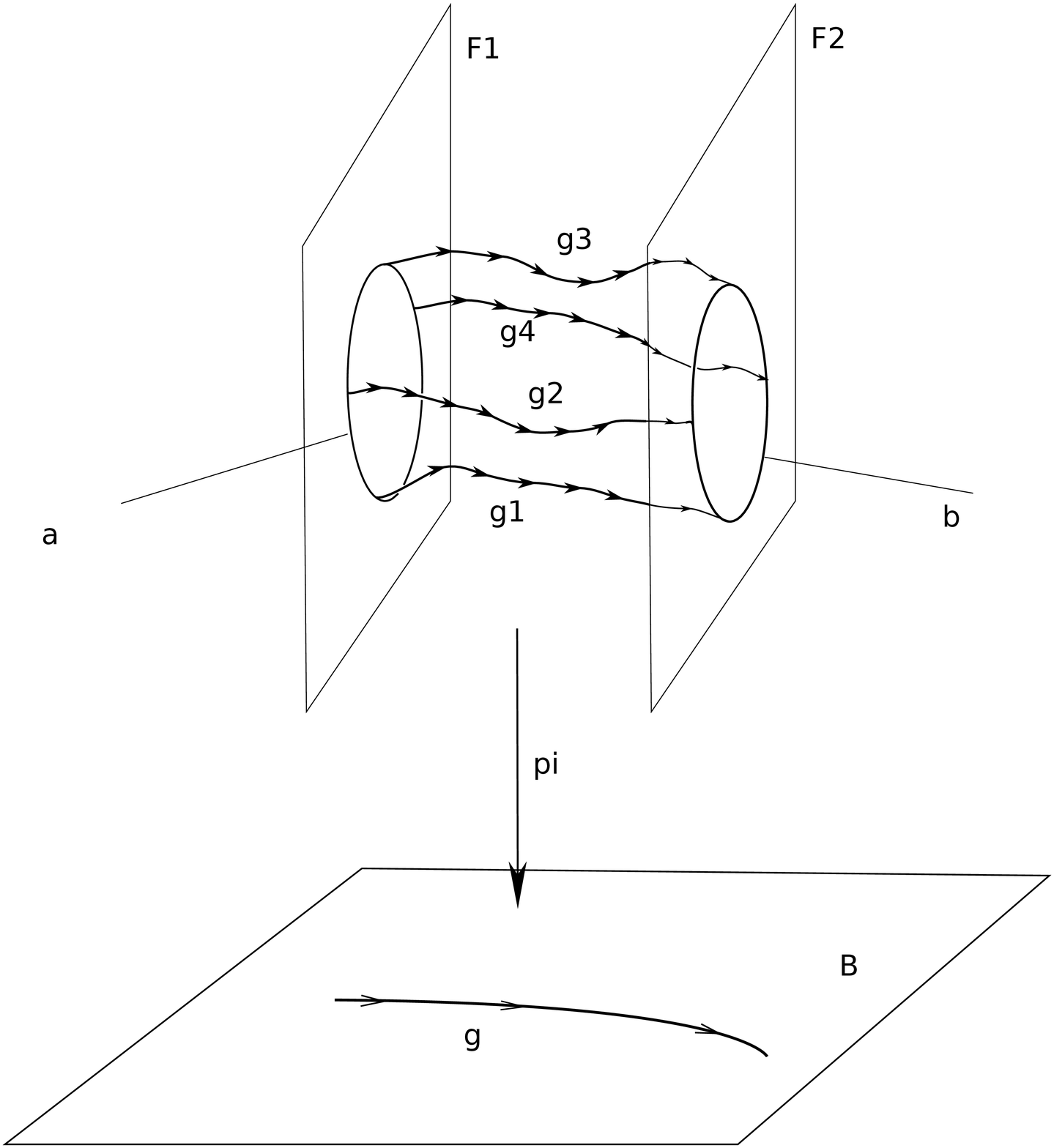}
\caption{L is generated by parallel transport}
\label{f4}
\end{figure}

\begin{proposition}\label{PPTlagtorusfibreposition}
Let $L\subset M$ be an embedded Lagrangian torus which is fibered by
the symplectic fibration $p \colon M \to B$. Then $L$ is invariant
under parallel transport along $\gamma=p(L)$ with respect to the
symplectic connection $\om$. 
\end{proposition}

\begin{proof}
At a point $x\in L$ we have the $\om$-orthogonal splitting
$T_xM=\HH_x\oplus V_x$, where $V_x:=\ker d_xp$ denotes the tangent
space to the fibre. Since $L$ is fibered by $p$, the subspace
$T_xL+V_x\subset T_xM$ generated by $T_xL$ and $V_x$ is
$3$-dimensional. The condition that $L$ is Lagrangian implies
$(T_xL)^\om=T_xL$, thus $T_xL\cap \HH_x = (T_xL)^\om\cap
(V_x)^\om = (T_xL+V_x)^\om$ is 1-dimensional. This $1$-dimensional
subspace therefore contains the horizontal lift of $\dot\gamma$
through $x$ and the propsition follows. 
\end{proof}

\begin{remark}
Let $N:=p^{-1}(\gamma)$ be the $3$-dimensional submanifold of $M$
formed by the fibres that meet the torus $L$. In the 
definition of being fibered by $p$ we did not require the torus to be
transverse to the fibres of $p$ in $N$. If, however, $L$ is Lagrangian,
then Proposition~\ref{PPTlagtorusfibreposition}
shows that we get this property for free.
\end{remark}

{\bf Monotone tori in $S^2\times S^2$. }
From now on we assume that
$$  
   M = S^2\times S^2
$$
and the symplectic form $\om$ is cohomologous to the product form 
$$
   \om_\std:=\ostd\oplus\ostd, 
$$
where $\ostd$ is the standard area form on $S^2$ normalized by
$\int_{S^2}\ostd=1$. In other words, we require that
$$
   \int_{S^2\times\pt}\om=\int_{\pt\times S^2}\om=1. 
$$
Moreover, we assume that the Lagrangian torus $L$ is {\em monotone},
i.e., its Maslov class $\mu$ (see~\cite{MS}) and its symplectic area satisfy
$$
   \mu(\sigma)=4\int_\sigma\om \text{ for all }\sigma\in\pi_2(M,L).
$$
Here the monotonicity constant must be equal to $4$ because the class
$A=[S^2\times\pt]\in\pi_2(M)$ has Maslov index $\mu(A)=2\la
   c_1(TM),A\ra = 2\la c_1(TS^2),[S^2]\ra = 4$.

\begin{lemma}\label{Lmonotonetorusembcurve}
Let $L\subset (M=S^2\times S^2,\omega)$ be a monotone Lagrangian torus
with $\omega$ cohomologous to $\omega_\std$. Let $p \colon M \to
B$ be a symplectic fibration over the surface $B$ such that $L$
is fibered by $p$. Then the loop $\gamma:=p(L)$ is an embedded
curve, i.e., it has no double points.  
\end{lemma}

\begin{proof}
Since $L$ is orientable, all its Maslov indices on $\pi_2(M,L)$ are
even integers. In view of the monotonicity constant $4$, this
implies that the symplectic area of each embedded symplectic disk $D\subset M$
with boundary on $L$ must be a positive multiple of $1/2$. If
$\gamma$ had a double point $b$, then the solid torus $T$ from
Definition~\ref{DLfibered} would intersect the fibre $p^{-1}(b)$ in
two disjoint symplectic disks, which is impossibe because the fibre
has symplectic area $1$.  
\end{proof}

\begin{remark}
(a) For a smooth fibration $p:S^2\times S^2\to B$ over a surface $B$,
  both the fibres and the base are diffeomorphic to $S^2$. Indeed,
  denoting a fibre by $F$, the homotopy exact sequence $\pi_2(F)\to
  \pi_2(S^2\times S^2)\to\pi_2(B)$ implies
  $\pi_2(F)\cong\pi_2(B)\cong\Z$, so $F$ and $B$ must be diffeomorphic
  to $S^2$ or $\R P^2$. Since by the product formula for the Euler
  characteristic $\chi(F)\chi(B)=\chi(S^2\times S^2)=4$, both $F$ and
  $B$ must be diffeomorphic to $S^2$. 

(b) For a {\em monotone} Lagrangian torus $L$ in $M=S^2\times S^2$,
  the third condition in Definition~\ref{DLfibered} is actually a
  consequence of the first two. To see this, note first that in the
  proof of Lemma~\ref{Lmonotonetorusembcurve} we can rule out the
  double point $b$ without reference to the solid torus $T$: By the
  Jordan curve theorem, the two circles in $L\cap p^{-1}(b)$ would
  bound two disjoint symplectic disks in the fibre $p^{-1}(b)\cong
  S^2$, each of area a positive multiple of $1/2$, which again gives the desired 
  contradiction. Now an orientation of $L$ and a parametrization of
  the curve $\gamma\subset B$ induce via horizontal lifts of
  $\dot\gamma$ orientations of the circles $L_t:=L\cap p^{-1}(\gamma(t))$,
  and we define $T$ as the union of the disks $D_t\subset
  p^{-1}(\gamma(t))$ whose oriented boundary is $L_t$.  
\end{remark}

\subsection{Relative symplectic fibrations of $S^2\times S^2$}

We continue with the manifold $M=S^2\times S^2$ and the generators
$$
   A=[S^2\times\pt],\;B=[\pt\times S^2]\in H_2(M).
$$
Now we define the main object of study for this paper. 

\begin{definition} \label{DAMSF}
A \emph{relative symplectic fibration} on $M=S^2\times S^2$ is a
quintuple $\SS=(\mathcal{F},\omega,L,\Sigma,\Sigma')$ where
\begin{itemize}
\item $\mathcal{F}$ is a smooth foliation of $M$ by $2$-spheres in homology class $B$;
\item $\omega$ is a symplectic form on $M$ making the leaves of
  $\mathcal{F}$ symplectic with $\omega(A)=\omega(B)=1$; 
\item $\Sigma,\Sigma'$ are disjoint symplectic submanifolds in class
  $A$ which are transverse to all the leaves of $\mathcal{F}$, so in
  particular the projection $p \colon M \to \Sigma$ sending each leaf
  to its unique intersection point with $\Sigma$ defines a symplectic fibration;
\item $L\subset M$ is an embedded monotone Lagrangian torus fibered by $p$;
\item $\Sigma'$ is disjoint from the solid torus $T$ with $\partial
  T=L$ in Definition~\ref{DLfibered}; 
\item $\Sigma$ intersects each fibre $p^{-1}(\gamma(t))$ in
  the interior of the disk $T\cap p^{-1}(\gamma(t))$. 
\end{itemize}
\end{definition}

Note that for a monotone Lagrangian torus $L$ fibered by $p:M\to B$
there always exist disjoint smooth sections $\Sigma,\Sigma'$ of $p$
with $\Sigma'$ disjoint from the solid torus $T$ and $\Sigma\cap
p^{-1}(\gamma)$ contained in the interior of $T$. The crucial
condition in Definition~\ref{DAMSF} is that these sections can be
chosen to be symplectic. 

\begin{definition}\label{DisoAMSF}
(a) A {\em homotopy} of relative symplectic fibrations is a smooth
1-parametric family 
$$
   \SS_t=(\mathcal{F}_t,\omega_t,L_t,\Sigma_t,\Sigma_t'),\qquad t\in[0,1]
$$
of relative symplectic fibrations.

(b) The group $\Diff_\id(M)$ of diffeomorphisms $\phi:M\to M$ inducing
the identity on the second homology group $H_2(M)$ (and hence on all
homology groups) acts on relative
symplectic fibrations by push-forward  
$$
   \phi(\SS) := \Bigl(\phi(\mathcal{F}),\phi_\ast
   \omega,\phi(L),\phi(\Sigma),\phi(\Sigma')\Bigr).
$$
Two relative symplectic fibrations $\SS$ and $\wt\SS$ are called
\emph{diffeomorphic} if $\wt\SS=\phi(\SS)$ for a diffeomorphism $\phi$
of $M$ (which then necessarily belongs to $\Diff_\id(M)$). 

(c) Two relative symplectic fibrations $\SS$ and $\wt\SS$ on $M$ are called
{\em deformation equivalent} if there exists a diffeomorphism
$\phi\in\Diff_\id(M)$ such that $\phi(\SS)$ is homotopic to $\wt\SS$. 
\end{definition}

\begin{remark}
(a) Note that a diffeomorphism $\phi\in\Diff_\id(M)$ intertwines the
  symplectic connections of $\SS$ and $\phi(\SS)$ and their parallel
  transports. For example, $\SS$ has trivial holonomy 
  iff $\phi(\SS)$ does.  

(b) It is easy to see that deformation equivalence is an equivalence
  relation. Moreover, $\SS$ is deformation equivalent to $\wt\SS$ iff
  there exists a homotopy $\SS_t$ such that $\SS_0=\SS$ and $\SS_1$ is
  diffeomorphic to $\wt\SS$.

(c) Note that in the above definition nothing is said about the
  isotopy class of the diffeomorphism $\phi$. In fact, it is an open
  problem whether every $\phi\in\Diff_\id(M)$ is isotopic to the
  identity, so we do not know whether diffeomorphic relative symplectic
  fibrations are homotopic in general. However, by a theorem of Gromov
  (see Theorem~\ref{G85} below), two diffeomorphic relative symplectic
  fibrations {\em with the same symplectic form $\om_\std$} are homotopic. 
  This result will be crucial at the end of the proof of our main
  theorem. 
\end{remark}

\subsection{The standard relative symplectic fibration}

The {\em standard relative symplectic fibration}
$\SS_\std=(\FF_\std,\om_\std,L_\std,S_0,S_\infty)$ of $S^2\times S^2$
consists of the following data:
\begin{itemize}
\item $\FF_\std$ is the foliation by the fibres $\{z\}\times S^2$ of
  the projection $p_1:S^2\times S^2\to S^2$ onto the first factor;
\item $\om_\std=\ostd\oplus\ostd$ is the standard symplectic form;
\item $S_0=S^2\times \{S\}$ and $S_\infty=S^2\times\{N\}$, where
  $N,S\in S^2$ are the north and south poles;
\item $L_\std=E\times E$ is the Clifford torus, i.e.~the product of
  the equators in the base and fibre;
\item $T_\std=E\times D_\lh$, where $D_\lh\subset S^2$ denotes the
  lower hemisphere, is the solid torus bounded by $L_\std$.  
\end{itemize}

The main goal of this paper will be to deform a given relative
symplectic fibration to the standard one (see Theorem~\ref{thm:fib}
below). For later use, let us record the relative homology and
homotopy groups of the Clifford torus. 

\begin{lemma}\label{lem:homology}
For the Clifford torus $L_\std\subset S^2\times S^2$ the second relative
homotopy/homology group 
$$
   \pi_2(S^2\times S^2,L_\std)\cong H_2(S^2\times S^2,L_\std)\cong
   H_2(S^2\times S^2)\oplus H_1(T^2)$$
is free abelian generated by 
$$
   [S^2\times\pt],\ [\pt\times S^2],\ [D_\lh\times\pt],\ [\pt\times D_\lh].
$$ 
\end{lemma}

\begin{proof}
The long exact sequences of the pair $(M=S^2\times S^2,L=L_\std)$ and
the Hurewicz maps yield the commuting diagram 
\begin{equation*}
\begin{CD}
   0 @>>> \pi_2(M) @>>> \pi_2(M,L) @>>> \pi_1(L) @>>> 0 \\
   @VVV @VV{\cong}V @VVV @VV{\cong}V @VVV \\
   H_2(L) @>{0}>> H_2(M) @>>> H_2(M,L) @>>> H_1(L) @>>> 0. 
\end{CD}
\end{equation*}
Here the first horizontal map in the lower row is zero because $L$
bounds the solid torus $T_\std=E\times D_\lh$ in $S^2\times S^2$, where
$E\subset S^2$ denotes the equator. Now the middle vertical map is an
isomorphism by the five lemma, and the generators of $H_2(M,L)$ are
obtained from the generators $[S^2\times\pt],\ [\pt\times S^2]$ of
$H_2(M)$ and $[E\times\pt],\ [\pt\times E]$ of $H_1(L)$. 
\end{proof}

\section{Standardisations}\label{sec:stand}

In this section we show how to deform a relative symplectic fibration
to make it split (in a sense defined below) near the symplectic sections $\Sigma,\Sigma'$ and
near one fibre $F$. In particular, the standardised fibration will have
trivial holonomy in these regions. This provides a convenient setup
for the discussion in Section~\ref{sec:killholonomy}. 

\subsection{Pullback by diffeomorphisms}

In this subsection, we show how to put a relative symplectic fibration
$\SS$ into a nicer form via pullback by diffeomorphisms. Note that
this is not really changing $\SS$ but just looking at it from a
different angle. We will see that using pullbacks we can either
standardise all data except the symplectic form $\om$, or all data
except the foliation $\FF$. So the nontriviality of a relative
symplectic fibration only arises from the interplay of $\om$ and
$\FF$, as measured by the holonomy of the corresponding symplectic
connection.  

In order to establish a clean picture of what can be achieved by
pullbacks, we will prove some results in stronger versions than what we
actually need in the sequel. 
\smallskip

{\bf Fixing the fibration. }
We begin with a useful characterisation of diffeomorphisms that are
trivial on homology. Recall the generators $A=[S^2\times\pt]$ and
$B=[\pt\times S^2]$ of $H_2(S^2\times S^2)$. 

\begin{lemma}\label{Pdiffrelfib}
A diffeomorphism $\phi$ of $S^2\times S^2$ is trivial on homology if
and only if it is orientation preserving and satisfies $\phi_\ast(B)=B$.
\end{lemma}

\begin{proof}
The ``only if'' is clear, so let us prove the ``if''.  
Let us write $\phi_\ast(A)=mA+nB$ for integers $m,n$. 
Since $\phi$ is orientation preserving, it preserves the intersection
form on $H_2(S^2\times S^2)$ and we obtain
\begin{align*}
   1 &= \phi_\ast(A)\cdot\phi_\ast(B) = (mA+nB)\cdot B = m, \cr
   0 &= \phi_\ast(A)\cdot\phi_\ast(A) = (A+nB)\cdot(A+nB) = 2n. 
\end{align*}
This shows that $\phi_\ast A=A$, so $\phi_\ast$ is the identity on
$H_2(S^2\times S^2)$. 
\end{proof}

Our first normalisation result is 

\begin{prop}
[Fixing the fibration]\label{prop:fix-F}
Let $\SS=(\mathcal{F},\omega,L,\Sigma,\Sigma')$ be a relative
symplectic fibration of $M=S^2\times S^2$. Then there exists a
diffeomorphism $\phi\in\Diff_\id(M)$ such that
$\phi^{-1}(\SS)=(\mathcal{F_\std},\wt\omega,L_\std,S_0,S_\infty)$ for
some symplectic form $\wt\om$. 
\end{prop}

\begin{proof}
Consider the fibration $p:M\to\Sigma$ defined by $\FF$ and pick an
orientation preserving diffeomorphism $u:\Sigma\to S^2$. Then $u\circ
p:M\to S^2$ is a fibration by $2$-spheres. Since $\pi_1Diff_+(S^2)$ 
classifies $S^2$-bundles over $S^2$ and $Diff_+(S^2)$ deformation
retracts onto $SO(3)$, there are up to bundle isomorphism 
precisely two $S^2$-bundles over $S^2$: the trivial one
$p_1:S^2\times S^2\to S^2$ and a nontrivial one $X\to S^2$. The
total space $X$ of the nontrivial bundle is the blow-up of $\C P^2$
at one point, which is not diffeomorphic to $S^2\times S^2$ (e.g.~their
intersection forms differ). Thus the nontrivial bundle does not occur,
and we conclude that there exists a diffeomorphism $\phi:M\to M$ such 
the following diagram commutes:
\begin{equation*}
 \begin{CD}
    S^2\times S^2   @>\phi^{-1} >>   S^2\times S^2\\
    @VV{p }V                                   @VV{ p_1}V\\
    \Sigma                  @>u>>     S^2
 \end{CD}
\end{equation*} 
Moreover, the restriction of $\phi$ to each fibre is orientation
preserving, which implies the $\phi$ itself is orientation
preserving. Since the fibres of $p$ and $p_1$ all represent the
homology class $B$, it follows that $\phi_*B=B$ and thus
$\phi\in\Diff_\id(M)$ by Lemma~\ref{Pdiffrelfib}. After replacing 
$\SS$ by $\phi^{-1}(\SS)$, we may hence assume that $\FF=\FF_\std$ and
$p=p_1$. 

The section $\Sigma$ of $p_1:S^2\times S^2\to S^2$ can be
uniquely parametrized by $z\mapsto(z,f(z))$ for a smooth map $f:S^2\to
S^2$. After a preliminary isotopy we may assume that $f(z_0)=S$ equals
the south pole $S$ at a base point $z_0\in S^2$. 
Then $f$ represents a class in $\pi_2(S^2,S)$. Now by Hurewicz's
Theorem, $\pi_2(S^2,S))\cong H_2(S^2)$. Since
$[\Sigma]=A=[S^2\times\pt]$, the class of $f$ is trivial in
$H_2(S^2)$, and thus in $\pi_2(S^2,S)$, so that $f$ is
nullhomotopic. By smooth approximation, we find a smooth homotopy
$f_t$ from the constant map $f_0\equiv S$ to $f_1=f$. Now we use (a
fibered version of) the isotopy extension theorem to extend the family of
embeddings $S^2\times\{S\}\into M$, $(z,S)\mapsto(z,f_t(z))$ to a
family of fibre preserving diffeomorphisms $\phi_t:M\to M$ with
$\phi_0=\id$. After replacing $\SS$ by $\phi_1^{-1}(\SS)$, we may hence
assume that $\Sigma=S^2\times\{S\}=S_0$. Now we repeat the same
argument with $\Sigma'$ (this time it is even simpler because
$S^2\setminus\{S\}$ is contractible) to arrange
$\Sigma'=S^2\times\{N\}=S_\infty$. 

Now the torus $L$ is fibered by $p_1:S^2\times(S^2\setminus\{N,S\})\to
S^2$. By an isotopy of the base $S^2$ we can move the embedded curve
$p_1(L)$ to the equator $E\subset S^2$. Let $e(t)$, $t\in\R/\Z$, be a
parametrization of the equator $E$ and consider the loop of embedded
closed curves $\Lambda_t:=L\cap p_1^{-1}(e(t))$ in the fibre $S^2$. After a
further homotopy we may assume that $\Lambda_0=E$. Pick a
smooth family of diffeomorphisms $g_t:S^2\to S^2$, $t\in[0,1]$, such
that $g_0=\id$ and $g_t(E)=\Lambda_t$ for all $t$. Moreover, we can
arrange that $g_t(N)=N$ and $g_t(S)=S$ for all $t$. Then $g_1$
satisfies $g_1(E)=E$ as well as $g_1(N)=N$ and $g_1(S)=S$. We can
alter $g_t$ so that $g_1$ fixes $E$ pointwise. By a theorem of
Smale~\cite{Sma}, the group $\Diff(D^2,\p D^2)$ of diffeomorphsms of the
disk that are the identity near the boundary is contractible. 
So we can alter $g_t$ further (applying this to the upper and lower
hemispheres) so that $g_1=\id$. This may first destroy the conditions
$g_t(N)=N$ and $g_t(S)=S$, but they can be reinstalled by a further
alteration. Now we again use (a fibered version of) the isotopy
extension theorem to extend the embedding $E\times S^2\into M$,
$(e(t),w)\mapsto(e(t),g_t(w))$ to a fibre preserving diffeomorphism
$\phi:M\to M$ isotopic to the identity. Then $\phi^{-1}(\SS)$ has the
desired properties and the proposition is proved. 
\end{proof}

A similar (in fact, simpler) proof yields the following $1$-parametric
version of Proposition~\ref{prop:fix-F}. 

\begin{prop}
[Fixing the fibration -- parametric version]\label{prop:fix-F-par}
Let $\SS_t=(\FF_t,\om_t,$ 
$L_t,\Sigma_t,\Sigma_t')_{t\in[0,1]}$ be
a homotopy of relative
symplectic fibrations of $M=S^2\times S^2$. Then there exists an
isotopy of diffeomorphisms $\phi_t\in\Diff_\id(M)$ with $\phi_0=\id$ such that
$\phi_t^{-1}(\SS_t)=(\FF_0,\wt\omega_t,L_0,\Sigma_0,\Sigma_0')$ for
some family of symplectic forms $\wt\om_t$. \hfill$\square$
\end{prop}

{\bf Fixing the symplectic form. }
Our next result is an easy consequence of Moser's and Banyaga's
theorems. Since it will be used repeatedly in this paper, let us
recall the latter~\cite[Th\'eor\`eme II.2.1]{Ban} for future reference.  

%

\begin{thm}[Banyaga's isotopy extension theorem~\cite{Ban}]\label{thm:Banyaga}
Let $(M,\om)$ be a symplectic manifold and $\psi_t:M\to M$ a smooth
isotopy with $\psi_0=\id$ such that each $\psi_t$ is symplectic on a
neighbourhood of a compact subset $X\subset M$. Suppose that 
$\int_{\sigma}\psi_t^*\om$ is constant in $t$ for each $\sigma\in
H_2(M,X)$. Then there exists a symplectic isotopy $\phi_t$ with
$\phi_0=\id$ and $\phi_t|_X=\psi_t|_X$. 
\end{thm}

\begin{prop}
[Fixing the symplectic form -- parametric version]\label{prop:fix-om-par}
Let $\SS_t=(\mathcal{F}_t,\omega_t,L_t,\Sigma_t,\Sigma_t')_{t\in[0,1]}$ be
a homotopy of relative
symplectic fibrations of $M=S^2\times S^2$. Then there exists an
isotopy of diffeomorphisms $\phi_t\in\Diff_\id(M)$ with $\phi_0=\id$ such that
$\phi_t^{-1}(\SS_t)=(\wt\FF_t,\omega_0,L_0,\Sigma_0,\Sigma_0')$ for
some family of foliations $\wt\FF_t$. 
\end{prop}

\begin{proof}
First, Moser's theorem provides an isotopy of diffeomorphisms
$\phi_t:M\to M$ with $\phi_0=\id$ such that $\phi_t^*\om_t=\om_0$.  
After replacing $\SS_t$ by $\phi^{-1}(\SS_t)$, we may hence assume that
$\om_t=\om_0$ for all $t$. 

Next, consider the isotopy of submanifolds
$X_t:=L_t\amalg\Sigma_t\amalg\Sigma_t'$ of $(M,\om_0)$. 
Let us write (using the smooth isotopy extension theorem)
$X_t=\psi_t(X_0)$ for diffeomorphisms $\psi_t:M\to M$ 
with $\psi_0=\id$. Since $L_t$ is Lagrangian and
$\Sigma_t\amalg\Sigma_t'$ is symplectic, the Lagrangian and symplectic
neighbourhood theorems provide a modification of $\psi_t$ which is
symplectic on a neighbourhood of $X_0$. 

We claim that the symplectic area $\int_{\sigma_t}\om_0$ is constant
in $t$ for each $\sigma\in H_2(M,X_0)$, where we denote
$\sigma_t:=(\psi_t)_*\sigma\in H_2(M,X_t)$.   
To see this, note that the map $H_2(M,L_0)\to H_2(M,X_0)$ is
surjective because $H_1(\Sigma_0\amalg\Sigma_0')=0$. So it suffices to
prove the claim for classes $\sigma\in H_2(M,L_0)\cong\pi_2(M,L_0)$.  
Now recall that the Lagrangian tori $L_t$ are monotone with respect to
$\om_0$. Since the Maslov class $\mu(\sigma_t)$ of
$L_t$ is constant in $t$, so is the symplectic area
$\int_{\sigma_t}\om_0$ by monotonicity and the claim is proved. 

In view of the claim, $(X_0,\psi_t)$ satisfies the hypotheses of
Banyaga's Theorem~\ref{thm:Banyaga}. It follows that the smooth
isotopy $\psi_t$ can be altered to a symplectic isotopy $\phi_t$ with
$\phi_0=\id$ and $\phi_t(X_0)=X_t$. This is the desired isotopy in the
proposition. 
\end{proof}

A non-parametric version of Proposition~\ref{prop:fix-om-par} is
much more subtle and will be discussed in Section~\ref{ss:proof}.

\subsection{Standardisation near a fibre}

Let us pick the point $z_0:=(1,0,0)$ on the equator $E$ in the base, 
so that $F:=p_1^{-1}(z_0)$ is a fibre of $\FF_\std$ intersecting
$L_\std$ in the equator $E$. The following proposition shows that we
can deform every relative symplectic fibration to make the triple
$(\FF,\om,L)$ standard near $F$. 

\begin{prop}[Standardisation near a fibre]\label{prop:fibre-standard}
Every relative symplectic fibration $\SS=(\FF_\std,\om,L_\std,\Sigma,\Sigma')$
is homotopic 
to a fibration of the form 
$\wt\SS=(\FF_\std,\wt\om,L_\std,\wt\Sigma,\wt\Sigma')$ such that 
$\wt\om=\om_\std$ on a neighbourhood of the fibre $F$. 
\end{prop}

The proof is given in~\cite[Lemma 3.2.3]{Ivrii}. For convenience, we
recall the argument. It is based on two easy lemmas. 

\begin{lemma}\label{Pspecialsymp}
Let $E$ denote the equator and $D_\lh$ the lower hemisphere in
$S^2\subset \mathbb{R}^3$. Let $\sigma$ be a symplectic	
form on $S^2$ cohomologous to $\sigma_\std$ such that
$\int_{D_\lh}\sigma=\frac{1}{2}$. Then there exists an isotopy of 
diffeomorphisms $h_t:S^2\to S^2$ with $h_0=\id$ such that  
$h_t(E)=E$ for all $t$ and $h_1^\ast \sigma=\sigma_\std$. 
\end{lemma}

\begin{proof}
We apply Moser's theorem to $\sigma_t:=(1-t)\sigma_\std+t\sigma$ to find a
diffeotopy $f_t:S^2\to S^2$ with $f_t^*\sigma_t=\sigma_\std$. Since
$\int_{D_\lh}\sigma_t=1/2$ for all $t$, the $\sigma_\std$-Lagrangians
$f_t^{-1}(E)$ all bound disks of $\sigma_\std$-area $1/2$. Hence
Banyaga's Theorem~\ref{thm:Banyaga} yields
$\sigma_\std$-symplectomorphisms $g_t:S^2\to S^2$ with
$g_t(E)=f_t^{-1}(E)$ and $h_t:=f_t\circ g_t$ is the desired
diffeotopy. 
\end{proof}

\begin{lemma}\label{lem:coiso-nbhd}
Let $\om$ be a symplectic form on $M=S^2\times S^2$ compatible with
the standard fibration $p_1:M\to S^2$. Let $\delta\subset S^2$ be an
embedded closed arc passing through $z_0$. Then every
symplectomorphism $h:(F,\sigma_\std)\to(F,\om|_F)$ extends to a
diffeomorphism $\psi$ between neighbourhoods of	
$p_1^{-1}(\delta)$ preserving the fibres over $\delta$ and such that
$\psi^*\om=\om_\std$. 
\end{lemma}

\begin{proof}
Parallel transport in $N:=p_1^{-1}(\delta)$ with respect to $\om_\std$
from $p_1^{-1}(z)$ to $F$ and then with respect to $\om$
from $F$ to $p_1^{-1}(z)$ yields a fibre preserving diffeomorphism
$\phi:N\to N$ extending $h$ with $\phi^*(\om|_N)=(\om_\std)|_N$. By the
coisotropic neighbourhood theorem, $\phi$ extends to the desired
diffeomorphism $\psi$. 
\end{proof}

\begin{proof}[Proof of Proposition~\ref{prop:fibre-standard}]
By Lemma~\ref{Pspecialsymp}, there exists a diffeomorphism
$h:F\to F$ with $h(E)=E$ and $h^*(\om|_F)=\sigma_\std$.  
Let $\delta\subset E$ be an arc in the equator in the base
passing though $z_0$. By Lemma~\ref{lem:coiso-nbhd}, the
diffeomorphism $h$ extends to a diffeomorphism $\psi$ between neighbourhoods of
$p_1^{-1}(\delta)$ preserving the fibres over $\delta$ and such that
$\psi^*\om=\om_\std$. Thus the pullback fibration $\psi^*\FF_\std$ and the
pullback torus $\psi^{-1}(L_\std)$ coincide over $\delta$ with
$\FF_\std$ and $L_\std$, respectively (the latter holds because
$\psi^{-1}(L_\std)$ is obtained by parallel transport of
$\{z_0\}\times E$ along $\delta$). Since $h:F\to F$ is isotopic to the
identity through diffeomorphisms preserving $E$, we can restrict $\psi$ 
to a smaller neighbourhood $V=p_1^{-1}(V')$ of $F$ and extend it from
there to a diffeomorphism $\phi:M\to M$ preserving $L_\std$ (and which
equals the identity outside a larger neighbourhood of $F$). Then the
pullback $\phi^{-1}(\SS)$ satisfies $(\phi^*\om)|_V=\om_\std$ and
$\phi^{-1}(\FF)=\FF_\std$ on $p_1^{-1}(\delta)\cap V$.  

So far we have just put $\SS$ into a more convenient form by a
diffeomorphism, but now we will change it. Note that the fibres of  
$\FF_\std$ near $p_1^{-1}(\delta)\cap V$ are $C^1$-close to those of
$\phi^{-1}(\FF)$ and therefore symplectic for $\phi^*\om$. Hence we
can deform $\phi^{-1}(\FF)$ to a foliation $\ol\FF$, keeping it
$\om_\std$-symplectic and fixed on $p_1^{-1}(\delta)\cap V$ and outside
$V$, such that $\ol\FF=\FF_\std$ on a neighbourhood $U\subset V$ of
$F$. Thus we have constructed a homotopy from $\SS$ to
$\ol\SS=(\ol\FF,\phi^*\om,L_\std,\ol\Sigma,\ol\Sigma')$ such that 
$(\ol\FF,\phi^*\om)=(\FF_\std,\om_\std)$ on $U$. 
Finally, we apply Proposition~\ref{prop:fix-F-par} to this homotopy
outside the set $U$ to replace it by one with fixed foliation
$\FF_\std$ (as well as $L_\std$ which was fixed already). The end point
of this homotopy is the desired relative symplectic foliation
$\wt\SS$. 
\end{proof}

\begin{remark}
Even if in Proposition~\ref{prop:fibre-standard} the sections
$\Sigma,\Sigma'$ in $\SS$ are the standard sections $S_0,S_\infty$,
this will not be true for the sections in $\wt\SS$ unless the original 
sections were horizontal near $F$. This will be remedied in 
the following subsection. 
\end{remark}

\subsection{Standardisation near the sections}

Consider a relative symplectic fibration of the form
$\SS=(\FF_\std,\om,L_\std,S_0,S_\infty)$ with the the projections
$p_1,p_2:S^2\times S^2$ onto the two factors. 
 
\begin{definition}
We say that $\om$ is {\em split} on a neighbourhood 
$$
   W=(U_F\times S^2)\cup (S^2\times U_0)\cup (S^2\times U_\infty)
$$ 
of $F\cup S_0\cup S_\infty$ if there exist symplectic forms
$\sigma_0,\sigma_\infty$ on $S^2$ such that 
\begin{equation}\label{eq:split}
\begin{aligned}
   \omega &= p_1^*\sigma_0+p_2^*\sigma_\std\quad \text{on the set}\quad
   W_0=(U_F\times S^2) \cup (S^2\times U_0), \quad\text{and}\cr
   \omega &= p_1^*\sigma_\infty+p_2^*\sigma_\std \quad\text{on the set}\quad
   W_\infty=(U_F\times S^2)\cup(S^2\times U_\infty).
\end{aligned}
\end{equation}
\end{definition}

Here the forms $\sigma_0$ and $\sigma_\infty$ may differ, but
they agree on $U_F$. Note that if $\om$ is split, then in particular
the sections $S_\infty$ and $S_0$ are horizontal. Moreover, parallel
transport of the symplectic connection defined by $\om$ equals the
identity on the region where $\om$ is split.   

The following is the main result of this section. 

\begin{prop}[Standardisation near a fibre and the sections]\label{prop:section-standard}
Every relative symplectic fibration $\SS=(\FF_\std,\om,L_\std,S_0,S_\infty)$
is homotopic to one of the form
$\wt\SS=(\FF_\std,\wt\om,L_\std,S_0,S_\infty)$ such that
$\wt\om$ is split on a neighbourhood $W$ of $F\cup S_0\cup S_\infty$. 
\end{prop}

The proof of this proposition will occupy the remainder of this section. 
Standardisation near a symplectic section is more subtle than near a 
fibre because the section need not be horizontal, so it takes a large
deformation to make it symplectically orthogonal to the fibres. 

We first consider the local
situation in $\R^4\cong\C^2$ with the standard symplectic form $\Om_0=dx\wedge
dy+du\wedge dv$ in coordinates $z=x+iy$, $w=u+iv$. Let $S=\{w=f(z)\}$
be the graph over the $z$-plane of a smooth function $f=g+ih:\R^2\to\R^2$
with $f(0)=0$. Orient $S$ by projection onto the $z$-plane. The
pullback of $\Om_0$ under the embedding $F(z)=\bigl(z,f(z)\bigr)$ 
equals 
$$
   F^*\Om_0=dx\wedge dy+dg\wedge dh = (1+\det Df)dx\wedge dy.  
$$
Thus $S$ is symplectic (with the given orientation) iff 
$$
   1+\det Df>0.
$$ 
For a smooth function $\phi:[0,\infty)\to\R$ consider the new function
$$
   \tilde f(z):=\phi(|z|)f(z).
$$ 
We now derive the condition on $\phi$ such that the graph of $\tilde
f$ is symplectic. We will see that it suffices to do this for linear
maps $f$, so suppose that $f(z)=Az$ for a $2\times 2$ matrix
$A$. We compute for $r:=|z|>0$: 
\begin{align*}
   D\tilde f(z) &= \phi(r)Df(z) +
   \phi'(r)f(z)\left(\frac{z}{r}\right)^t = \phi(r)A +
   \frac{\phi'(r)}{r}Azz^t \cr
   &= A\Bigl(\phi(r)\Id+\frac{\phi'(r)}{r}zz^t\Bigr). 
\end{align*}
Since
\begin{align*}
   \det\Bigl(\phi(r)\Id+\frac{\phi'(r)}{r}zz^t\Bigr) 
   &= \det\left(\begin{matrix} \phi+\frac{\phi'}{r}x^2 &
   \frac{\phi'}{r}xy \\ \frac{\phi'}{r}xy & \phi+\frac{\phi'}{r}y^2
   \end{matrix}\right) \cr
   &= \phi^2+\frac{\phi\phi'}{r}(x^2+y^2) = \phi^2+r\phi\phi',
\end{align*}
we have $\det D\tilde f = (\phi^2+r\phi\phi')\det A$. This proves 

\begin{lemma}\label{lem:symp-condition}
Let $f(z)=Az$ be a linear function $\R^2\to\R^2$ with $1+\det A\geq
\eps>0$. Let $\phi:[0,\infty)\to\R$ be a smooth function with
$\phi(0)=\phi'(0)=0$. Then the graph of $\tilde f(z):=\phi(|z|)f(z)$
is symplectic provided that for all $r>0$,
\begin{equation}\label{eq:graph-symp}
   0\leq \phi(r)^2+r\phi(r)\phi'(r) < \frac{1}{1-\eps}.
\end{equation}
\end{lemma}

\begin{lemma}\label{lem:symp-function}
For every $0<\eps<1$ and $\delta>0$ there exists a smooth family of 
nondecreasing functions $\phi_s:[0,\infty)\to[0,1]$, $s\in[0,1]$,
satisfying  (\ref{eq:graph-symp}) such that $\phi_s(r)=s$ for
$r\leq\delta$ and $\phi_s(r)=1$ for $r\geq 2\delta/\sqrt{\eps}$.  
\end{lemma}

\begin{proof}
For $r>0$ define $\psi(r)$ by $\psi(r):=r^2\phi(r)^2$. Then
$\psi'=2r(\phi^2+r\phi\phi')$, so (\ref{eq:graph-symp}) is equivalent
to 
$$
  \psi'(r)<\frac{2r}{1-\eps}.
$$
This will be satisfied if $\psi$ solves the differential equation 
$$
   \psi'(r) = \frac{2r}{1-\eps/4}.
$$
Then $\psi(r)=r^2/(1-\eps/4)+c$ for some constant $c$ and 
$$
   \phi^2(r) = \frac{1}{1-\eps/4}+\frac{c}{r^2}.
$$
We fix the constant $c$ by $\phi(\delta)=0$ to
$c=-\delta^2/(1-\eps/4)$ and obtain
$$
   \phi^2(r)=\frac{1-\delta^2/r^2}{1-\eps/4}.
$$
This is an increasing function with $\phi(\delta)=0$ and $\phi(\gamma)=1$ at the point
$\gamma=2\delta/\sqrt{\eps}$. Now observe that if a solution of
(\ref{eq:graph-symp}) satisfies $\phi(r_0)\geq 0$ and $\phi'(r_0)\geq
0$ for some $r_0>0$, then we can decrease the slope to $0$ near $r_0$ and
extend $\phi$ by $\phi(r)=\phi(r_0)$ for $r\geq r_00$ (or $r\leq r_0$) to
a smooth solution of (\ref{eq:graph-symp}). Applying this procedure
at $r_0=\delta$ and $r_0=\gamma$ yields the desired function $\phi_0$
for $s=0$. For $s>0$, we obtain $\phi_s$ by smoothing the function
$\max(s,\phi_0)$. 
\end{proof}

\begin{lemma}\label{lem:symp-standard}
Let $\Lambda\subset\R^2$ be compact and
$(S^\lambda)_{\lambda\in\Lambda}$ be a smooth foliation of a region in
$(\R^4,\Om_0)$ by symplectic surfaces $S^\lambda$ intersecting
the symplectic plane $\{0\}\times\R^2$ transversely in $(0,\lambda)$. Then for every
neighbourhood $W\subset\R^4$ of $\{0\}\times\Lambda$ there exists a neighbourhood
$U\subset W$ of $\{0\}\times\Lambda$ and a family of foliations
$(S^\lambda_s)_{s\in[0,1],\lambda\in\Lambda}$ with the following
properties (see Figure~\ref{fig:Slambda}):
\begin{enumerate}
\item $S^\lambda_0=S^\lambda$ and $S^\lambda_s=S^\lambda$ outside $W$; 
\item $S^\lambda_s$ is symplectic and intersects $\{0\}\times\R^2$
transversely in $(0,\lambda)$;
\item $S^\lambda_1=\R^2\times\{\lambda\}$ in $U$.
\end{enumerate}
Moreover, for every $\lambda$ with $S^\lambda=\R^2\times\{\lambda\}$
in $W$ we have $S^\lambda_s=S^\lambda$ for all $s$. 
\end{lemma}

\begin{figure}[h!]
\centering
\psfrag{l}{}
\psfrag{R2}{$\{0\}\times\mathbb{R}^2$}
\psfrag{le}{after}
\psfrag{ls}{deformation}
\psfrag{Sn}{$S^\lambda_1$}
\psfrag{S}{$S^\lambda=S^\lambda_0$}
\psfrag{l}{$(\lambda,0)$}
\psfrag{ss}{$S^\lambda_s$}
\psfrag{s1}{$S^\lambda_1$}

 \includegraphics[width=12cm]{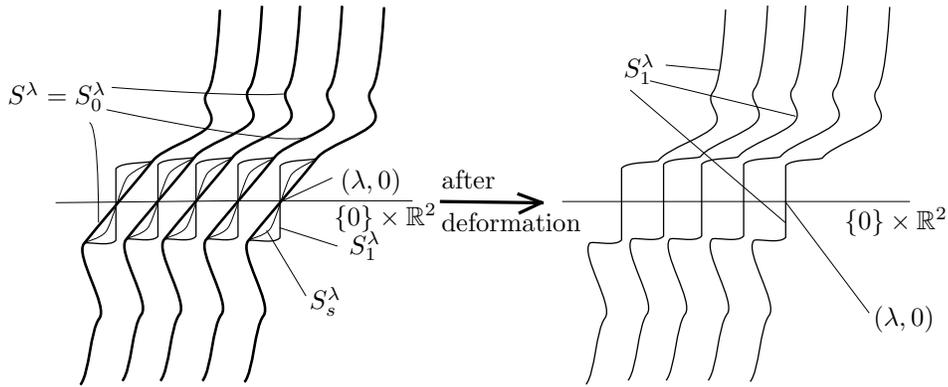}
 \caption{The family of foliations $S_s^\lambda$}
 \label{fig:Slambda}
\end{figure}

\begin{proof}
After shrinking $W$, we may assume that in $W$ each surface can be
written as a graph $S^\lambda=\{w=\lambda+f^\lambda(z)\}$ over the
$z$-plane with $f^\lambda(0)=0$. After a $C^1$-small perturbation of
the surfaces in $W$ (which keeps them symplectic) we may assume that
the $f^\lambda$ are linear functions $f^\lambda(z)=A^\lambda z$. Symplecticity implies $\det
A^\lambda>-1$. Since $\Lambda$ is compact, there exists an
$\eps>0$ with $\det A^\lambda\geq -1+\eps$ in $W$ for all
$\lambda$. Moreover, we may assume that the $\eps$-neighbourhood of
$\Lambda$ is contained in $W$. Pick $\delta>0$ so small that 
$2\delta/\sqrt{\eps}<\eps$. Let $\phi_s:[0,\infty)\to[0,1]$,
$s\in[0,1]$, be the functions of Lemma~\ref{lem:symp-function} and
define $f^\lambda_s(z):=\phi_{1-s}(|z|)f^\lambda(z)$. By
Lemma~\ref{lem:symp-condition}, the graph $S^\lambda_s$ of
$f^\lambda_s$ satisfies conditions (i)-(iii) of the proposition, where
$U$ is the $\delta$-neighbourhood of $\Lambda$. Note that if
$S^\lambda=\R^2\times\{\lambda\}$ for some $\lambda$, then
$f^\lambda(z)\equiv 0$ and thus $S^\lambda_s=S^\lambda$ for all $s$. 

It only remains to verify that the surfaces
$(S^\lambda_s)_{\lambda\in\Lambda}$ form a foliation for each $s$, or
equivalently, that the map $F_s:B^2(\eps)\times\Lambda\to\R^4$,
$$
   F_s(z,\lambda) := \bigl(z,\lambda+f^\lambda_s(z)\bigr) =
   \bigl(z,\lambda+\phi_{1-s}(|z|)A^\lambda z\bigr),			
$$
is an embedding. For injectivity, suppose that
$F_s(z,\lambda)=F_s(z',\lambda')$. Then $z=z'$ and
$\lambda-\lambda'=-\phi_{1-s}(|z|)(A^\lambda-A^{\lambda'})z$. This implies 
$$
   |\lambda-\lambda'|\leq\|A^\lambda-A^{\lambda'}\|\;|z| \leq
   \eps\|A^\lambda-A^{\lambda'}\|.  
$$
Since $A^\lambda$ depends smoothly on $\lambda$, there exists a
constant $C$ such that $\|A^\lambda-A^{\lambda'}\|\leq
C|\lambda-\lambda'|$. For $\eps<1/C$ it follows that
$\lambda=\lambda'$. For the immersion property, consider the
differential 
$$
   DF_s(z,\lambda) = \left(\begin{matrix} \Id & 0 \\ 
   D_zf^\lambda_s & \Id+B_s
   \end{matrix}\right),\qquad B_s = \frac{\p f^\lambda_s}{\p\lambda}.  
$$ 
This is invertible iff the matrix 
$$
   \Id+B_s = \Id+\phi_{1-s}(|z|)\frac{\p
   A^\lambda}{\p\lambda}z 
$$
is invertible. By smoothness in $\lambda$, there exists a constant $C$
with $\|\frac{\p A^\lambda}{\p\lambda}z\|\leq C|z|$. Then for 
$\eps<1/C$ we get 
$$
   \|\phi_{1-s}(|z|)\frac{\p A^\lambda}{\p\lambda}z\| \leq C|z| \leq C\eps
   <1,  
$$
which implies invertibility of $\Id+B_s$.
\end{proof}

\begin{proof}[Proof of Proposition~\ref{prop:section-standard}]
We deform the given relative symplectic fibration
$\SS=(\FF_\std,\om,L_\std,S_0,S_\infty)$ in 4 steps. 

{\em Step 1. }By Proposition~\ref{prop:fibre-standard}, $\SS$ 
is homotopic to 
$\wt\SS=(\FF_\std,\wt\om,L_\std,\wt\Sigma,\wt\Sigma')$ such that 
$\wt\om=\om_\std$ on a neighbourhood of the fibre $F=\{z_0\}\times
S^2$. The sections $\wt\Sigma,\wt\Sigma'$ intersect the fibre in
points $(z_0,q)$ and $(z_0,q')$. 
After pulling back $\wt\SS$ by a symplectomorphism
$(z,w)\mapsto(z,g(w))$, where $g:S^2\to S^2$ is a Hamiltonian
diffeomorphism preserving the equator and mapping the south pole $S$
to $q$ and the north pole $N$ to $q'$, we may assume in addition that
$\wt\Sigma\cap F=(z_0,S)$ and $\wt\Sigma'\cap F=(z_0,N)$. 

{\em Step 2.} Consider the symplectic section $\wt\Sigma$. By
Lemma~\ref{lem:symp-standard} (with $\Lambda=\{0\}$, $S^0=\wt\Sigma$ and
$F=\{0\}\times\R^2$ in local coordinates), we can deform $\wt\Sigma$ such 
that it agrees with $S_0=S^2\times\{S\}$ near $\wt\Sigma\cap F$. 
Since the section $\wt\Sigma$ is isotopic to $S_0$, there exists  
a diffeomorphism of $S^2\times S^2$, isotopic to the identity
and fixed near $F$, mapping $S_0$ to $\wt\Sigma$. After pulling back everything
by this diffeomorphism, we may assume that $\wt\Sigma=S_0$.
As in the proof of Proposition~\ref{prop:fibre-standard}, using the
symplectic neighbourhood theorem, by pulling back by an isotopy of
$S^2\times S^2$ fixed near $F$ we can arrange in addition that
$\wt\om=\om_\std$ near $\wt\Sigma=S_0$ (but the foliation becomes
non-standard). The same arguments apply to the 
other section $\wt\Sigma'$. Altogether, we have shown that 
$\wt\SS$ is homotopic to a relative symplectic fibration of the form
$\wh\SS=(\wh\FF,\wh\om,L_\std,S_0,S_\infty)$ with the following properties:
$\wh\om=\om_\std$ and $\wh\FF=\FF_\std$ near the fibre $F=\{z_0\}\times S^2$,
and $\wh\om=\om_\std$ near the symplectic sections $S_0$ and $S_\infty$. 

{\em Step 3. }Next, we adjust the foliation $\wh\FF$ near $S_0\cup S_\infty$. 
Consider first $S_0$. 
Take a compact subset $\Lambda\subset S^2\setminus\{z_0\}$ such that
$\wh\FF=\FF_\std$ on a neighbourhood of $(S^2\setminus{\rm int}\Lambda)\times
S^2$. We identify $\Lambda$ with a subset of
$(\R^2,dx\wedge dy)$, and a neighbourhood of $\Lambda\times\{S\}$ in
$S^2\times S^2$ with a neighbourhood $W$ of $\{0\}\times\Lambda$ in
$(\R^4,\Om_0)$, by a symplectomorphism of the form
$(z,w)\mapsto(f(w),g(z))$. Under this identification, $\wh\FF$
corresponds to a symplectic
foliation of $W$ transverse to $\{0\}\times\Lambda$ and
standard near $\p\Lambda\times\R^2$. By 
Lemma~\ref{lem:symp-standard}, $\wh\FF$ can be deformed in $W$,
keeping it fixed near $\p\Lambda\times\R^2$, to a symplectic foliation
that is standard on a neighbourhood $U$ of $\{0\}\times\Lambda$ in
$\R^4$. Transfering back to $S^2\times S^2$ and performing the same
construction near $S_\infty$, we have thus deformed $\wh\SS$ to a
relative symplectic fibration $\ol\SS=(\ol\FF,\ol\om,L_\std,S_0,S_\infty)$
satisfying $\ol\om=\om_\std$ and $\ol\FF=\FF_\std$ near the set $F\cup
S_0\cup S_\infty$. This was the main step. It only remains to deform
$\ol\FF$ back to $\FF_\std$. 

{\em Step 4. }By construction, the foliation $\ol\FF$ is obtained in
steps 2 and 3 from $\FF_\std$ by a homotopy $\FF_t$ with $\FF_0=\FF$
and $\FF_1=\ol\FF$ which is fixed outside $V\setminus V_F$, for some
neighbourhoods $V$ of $S_0\cup S_\infty$ and $V_F$ of $F$. 
So we can write $\FF_t=\phi_t(\FF_\std)$ for diffeomorphisms with
$\phi_0=\id$ and $\phi_t=\id$ outside $V\setminus V_F$. Since $\ol\FF$ agrees with
$\FF_\std$ on a smaller neighbourhood $V_0\cup V_\infty$ of $S_0\cup S_\infty$, 
the diffeomorphisms $\phi_t$ can be chosen of the form
$(z,w)\mapsto(z,f_t(w))$ on $V_0$ and $(z,w)\mapsto(z,g_t(w))$ on
$V_\infty$ for diffeomorphisms $f_t,g_t$ of $S^2$. Now the homotopy
$\phi_t^{-1}(\ol\FF)$ connects $\ol\FF$ to the symplectic fibration
$\phi_1^{-1}\ol\SS=(\FF_\std,\phi_1^*\ol\om,L_\std,S_0,S_\infty)$,
where $\phi_1^*\ol\om$ is split near $F\cup S_0\cup S_\infty$. This
concludes the proof of Proposition~\ref{prop:section-standard}. 
\end{proof}

\begin{remark}\label{rem:section-standard}
Replacing Step 4 of the preceding proof by a more careful deformation 
of the foliation $\ol\FF$ (not by diffeomorphisms but keeping it
symplectic for $\ol\om$), we could arrange $\wt\om=\om_\std$ near
$F\cup S_0\cup S_\infty$ in Proposition~\ref{prop:section-standard}. 
As the class of split forms is better suited for the modifications in
the next section, we content ourselves with making $\wt\om$ split near
$F\cup S_0\cup S_\infty$.  
\end{remark}

\section{Killing the holonomy}\label{sec:killholonomy}

In this section we will deform a relative symplectic fibration to
kill all the holonomy and conclude the proof of the main theorem. A
crucial ingredient is the inflation procedure from~\cite{LalMcD}. 

\subsection{Setup}\label{ss:setup}

Recall that $\FF_\std$ is the foliation on $S^2\times S^2$ given
by the fibres of the projection $p_1$ onto the first factor,
$S_0=S^2\times \left\lbrace S\right\rbrace$ and $S_\infty=S^2\times
\left\lbrace N\right\rbrace$ are the standard sections, 
$F=p_1^{-1}(z_0)$ is the fibre over the point $z_0=(1,0,0)$, and the
Clifford torus $L_\std=E\times E$ is the product of the equators.
In the following, we identify $S_0$ with the base $S^2$ of the
projection $p_1$, i.e., we identify $p_1$ with the map
$(z,w)\mapsto(z,S)$ sending each fibre to its intersection with
$S_0$.

Our starting point is a relative symplectic fibration $\SS=(\FF_\std,
\omega, L_\std,S_0,S_\infty)$ as provided by Proposition~\ref{prop:section-standard}
such that $\om$ is split on a neighbourhood $W=(U_F\times S^2)\cup
(S^2\times U_0)\cup (S^2\times U_\infty)$ of $F\cup S_0\cup S_\infty$.   
In particular the sections $S_0,S_\infty$ are horizontal for the
symplectic connection. After pulling back $\SS$ by a diffeomorphism of
the form $(z,w)\mapsto (\phi(z),w)$ (keeping the same notation), we
may replace $U_F$ by the ball 
$$
   B:=\left\lbrace (x,y,z)\in S^2\mid x\geq -1/\sqrt{2}\right\rbrace,
$$
so that $\om$ now satisfies
\begin{equation}\label{eq:split2}
\begin{aligned}
   \omega &= p_1^*\sigma_0+p_2^*\sigma_\std\quad \text{on the set}\quad
   W_0=(B\times S^2) \cup (S^2\times U_0), \quad\text{and}\cr
   \omega &= p_1^*\sigma_\infty+p_2^*\sigma_\std \quad\text{on the set}\quad
   W_\infty=(B\times S^2)\cup(S^2\times U_\infty).
\end{aligned}
\end{equation}
Consider the usual spherical coodinates $(\lambda,\mu) \in
[-\frac{\pi}{2},\frac{\pi}{2}]\times [0,2\pi]$ on the base $S^2$
centered at $z_0$. Thus $\lambda$ denotes the latitude and $\mu$ the
meridian, and $z_0$ lies at $(\lambda,\mu)=(0,0)$; see Figure~\ref{f42}.

\begin{figure}[h!]
\centering
\psfrag{B}{$B$}
\psfrag{C}{$C^\lambda$}
\psfrag{z0}{$z_0$}

 \includegraphics[width=5cm]{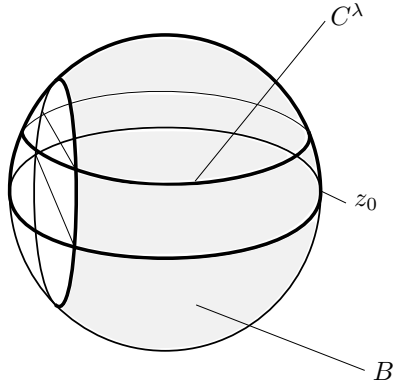}
 \caption{Circles of latitude and the set $B$}
 \label{f42}
\end{figure}

Denote by $C^\lambda$ the circle of latitude $\lambda$ in the base and
by $\phi^\lambda$ the symplectic parallel transport around $C^\lambda$
parametrised by $\mu\in[0,2\pi]$. Since the starting and end points of
the parametrisation of $C^\lambda$ are contained in
$B$ for all $\lambda$ and the symplectic form
$\omega$ equals $p_1^*\sigma_0+p_2^*\sigma_\std$ over
$B$, we can regard $\phi^\lambda$ as living in
$Symp(S^2,\ostd)$ for all $\lambda$. Moreover, the maps $\phi^\lambda$ have
the following two properties:
\begin{enumerate}
\item Since $C^\lambda\subset B$ for all $|\lambda|\geq \frac{\pi}{4}$
  and the form $\om$ is split on $B\times S^2$, we have $\phi^\lambda=\id$ for
  $|\lambda|\geq \frac{\pi}{4}$. 
\item Since $\om$ is split on $S^2\times(U_0\cup U_\infty)$, each
  $\phi^\lambda$ restricts to the identity on $U_0\cup U_\infty$.  
\end{enumerate}
Under stereographic projection $S^2\setminus\{N\}\to\C$ from the north
pole $N$, the standard symplectic form on $S^2$ corresponds to the
form $\ostd=\frac{r}{\pi(1+r^2)^2}dr\wedge d\theta$ in polar
coordinates on $\C$. We pick a closed annulus 
$$
   A=\left\lbrace z\in \mathbb{C}\;\bigl|\; a\leq |z|\leq
   b\right\rbrace\subset\C\cong S^2\setminus\{N\}
$$ 
with $a>0$ so small and $b>a$ so large that $\p A\subset U_0\cup
U_\infty$. According to properties (i) and (ii) above, parallel
transport along $C^\lambda$ then defines maps $\phi^\lambda \in
\Symp(A,\partial A,\ostd)$ (i.e., symplectomorphisms that equal the
identity near $\p A$, cf.~Appendix~\ref{app:diff}) that 
equal the identity for $|\lambda|\geq \frac{\pi}{4}$. In particular,
$[-\frac{\pi}{2},\frac{\pi}{2}]\ni\lambda \mapsto \phi^\lambda$
defines a loop in the identity component $Symp_0(A,\partial A,\ostd)$. 
Consider the loop of inverses 
$$
   \psi^\lambda=(\phi^\lambda)^{-1}.
$$ 
Since $L_\std=E\times E$ is invariant under parallel transport, the
map $\phi^0$, and thus $\psi^0$, preserves the equator $E$. 

\subsection{A special contraction}

According to Proposition~\ref{PHE}, the loop $\psi^\lambda$ is
contractible in $\Symp_0(A,\p A,\ostd)$. However, in order for the inflation
procedure below to work, we need a special contraction
$\psi^\lambda_s$ with the property that $\psi^0_s(E)=E$ for all
$s\in[0,1]$. Here we identify the equator $E$ in $S^2$ via
stereographic projection with the circle $E=\left\lbrace
|z|=1\right\rbrace \subset A$. 

\begin{prop}\label{Tspeccontr}
There exists a smooth contraction $\psi^\lambda_s\in\Symp_0(A,\p
A,\ostd)$ of the loop $\psi^\lambda$, with $(s,\lambda) \in [0,1]\times
[\frac{-\pi}{2},\frac{\pi}{2}]$, such that: 
\begin{enumerate}
\item $\psi^\lambda_0=\id$ and $\psi^\lambda_1=\psi^\lambda$ for all $\lambda$;
\item $\psi_s^\lambda=\id$ for $|\lambda|\geq \frac{\pi}{4}$ and all $s$;
\item $\psi^\lambda_s$ is constant in $s$ near $s=0$ and $s=1$;
\item $\psi^0_s(E)=E$ for all $s$.
\end{enumerate}
\end{prop}

\begin{proof}
Since the holonomy $\psi^0$ along the equator in the base preserves
the equator $E$ in the fibre, Lemma~\ref{L1} provides a path
$\alpha(t)\in\Symp_0(A,\partial A,\ostd)$ from the identity to
$\psi^0$ which preserves $E$ for all $t$. We split the loop
$\psi^\lambda$ into two paths $\delta_1:=\left\lbrace  \psi^\lambda
\right\rbrace_{\lambda \in [-\frac{\pi}{2},0]}$ and
$\delta_2:=\left\lbrace  \psi^\lambda \right\rbrace_{\lambda \in
  [0,\frac{\pi}{2}]}$. Using these, we define two loops $\gamma_1:=\delta_1\ast
\bar{\alpha}$ and $\gamma_2:=\alpha \ast \delta_2$, where $\ast$
means concatenation of paths and $\bar{\alpha}$ denotes the path
$\alpha$ traversed in the opposite direction; see Figure~\ref{f44}. 
By Proposition~\ref{PHE}, these loops are contractible in
$\Symp_0(A,\partial A,\ostd)$, so we can fill them by half-disks
$D_1,D_2$ in $\Symp_0(A,\partial A,\ostd)$. Gluing these half-disks 
along $\alpha$ yields a map $\vartheta: D\to \Symp_0(A,\partial
A,\ostd)$ from the unit disk $D\subset\C$ which restricts to the loop
$\psi^\lambda$ on $\p D$ (starting and ending at $-i$) and to the
path $\alpha$ on the imaginary axis. The composition of $\vartheta$
with the map 
$$
   \eta:[0,1]\times[-\frac{\pi}{2},\frac{\pi}{2}]\to D,\qquad 
   (s,\lambda)\mapsto (s-1)i+se^{i(2\lambda+\pi/2)}
$$ 
(see Figure~\ref{circlesoflatitude2}) then has properties
(i) and (iv) of the proposition. By smoothing and reparametrisation we
finally arrange properties (ii) and (iii) to obtain the desired contraction
$\psi^\lambda_s$. 
\begin{figure}[h!]
 \centering
\psfrag{alpha}{$\alpha$}
\psfrag{id}{$\id$}
\psfrag{p0}{$\psi^0$}
\psfrag{g1}{$\gamma_1$}
\psfrag{g2}{$\gamma_2$}
\psfrag{D1}{$D_1$}
\psfrag{d1}{$\delta_1$}
\psfrag{d2}{$\delta_2$}
\psfrag{D2}{$D_2$}
\psfrag{psi}{$\psi^\lambda$} 

 \includegraphics[width=11cm]{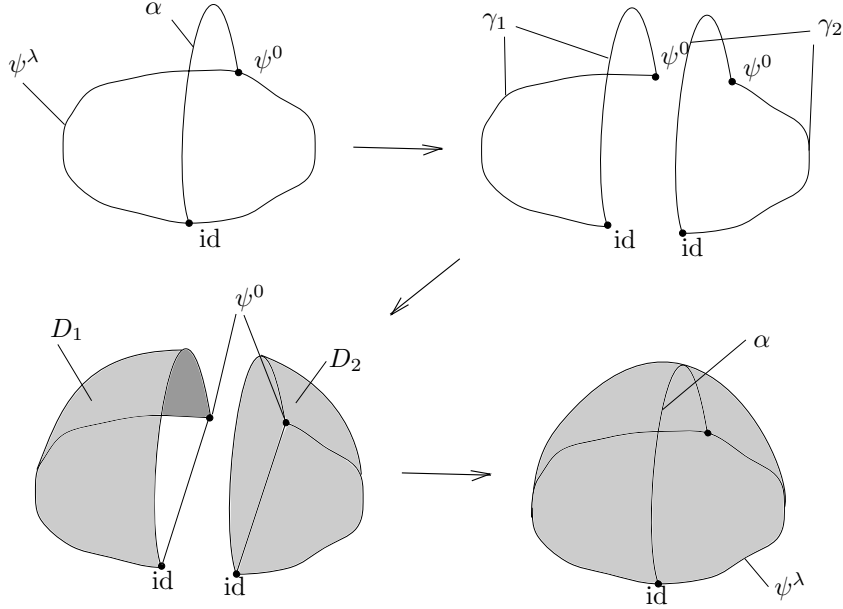}
 \caption{Construction of the special contraction $\psi^ \lambda_s$}
 \label{f44}
\end{figure}

\begin{figure}[h!]
 \centering
 \psfrag{B}{$\eta$}
 \psfrag{Bs}{$s$}
 \psfrag{Be}{$1$}
 \includegraphics[width=14cm,bb=0 -1 1777 793]{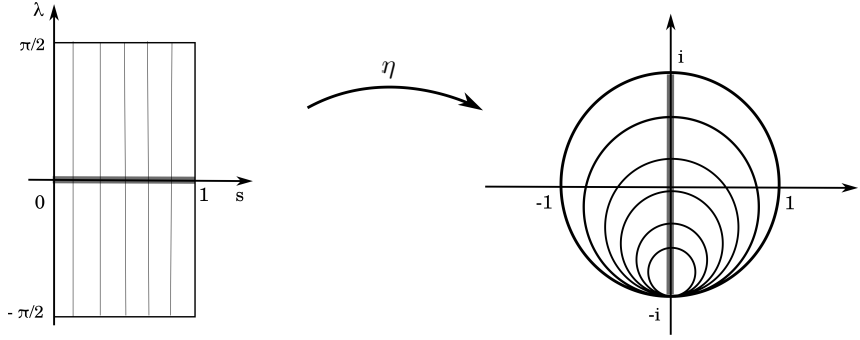}
\caption{The reparametrisation $\eta$}
 \label{circlesoflatitude2}
\end{figure}

\end{proof}

\subsection{A special Hamiltonian function}\label{ConstrH}

Next, we construct a family of time-dependent Hamiltonians generating
the contraction $\psi_s^\lambda$ of the previous subsection. 
We begin with a simple lemma. 

\begin{lemma}\label{LformulaHtFt}
Let $(M,\omega=d\lambda)$ be an exact symplectic manifold. Let
$\phi_t:M\to M$ be a symplectic isotopy starting at $\phi_0=\id$ 
generated by the time-dependent vector field $X_t$, i.e
$\frac{d}{dt}\phi_t=X_t\circ \phi_t$. 
Then $\iota_{X_t}\omega=dH_t$ for a smooth family of 
functions $H_t \colon M \to \mathbb{R}$ if and only if $\phi_t^\ast
\lambda-\lambda=dF_t$ for a smooth family of functions $F_t \colon M
\to \mathbb{R}$. Moreover, $F_t$ and $H_t$ are related by the equations 
$$
   F_t=\int_0^t(H_s+\iota_{X_s}\lambda)\circ \phi_s\,ds,\qquad 
   H_t=\dot F_t \circ \phi_t^{-1}-\iota_{X_t}\lambda.
$$
\end{lemma}

\begin{proof}
Assume first that $\iota_{X_t}\omega=dH_t$. Then 
\begin{gather*}
\phi_t^\ast \lambda-\lambda=\int_0^t\frac{d}{ds}(\phi_s^\ast
\lambda) ds=\int_0^t\phi_s^\ast (L_{X_s})\lambda ds \cr
=\int_0^t \phi_s^\ast (\iota_{X_s}d\lambda+d\iota_{X_s}\lambda)ds=d\int_0^t
(H_s+\iota_{X_s}\lambda)\circ \phi_s ds
\end{gather*}
shows that $\phi_t^\ast \lambda-\lambda=dF_t$ holds for $F_t:=\int_0^t
(H_s+\iota_{X_s}\lambda)\circ \phi_s ds$. Conversely, if $\phi_t^\ast
\lambda-\lambda=dF_t$, then we differentiate this equation to obtain 
$$
   d\dot F_t = \frac{d}{dt}(\phi_t^\ast \lambda) =
   \phi_t^\ast(d\iota_{X_t}\lambda+\iota_{X_t}d\lambda),
$$
which shows that $i_{X_t}d\lambda=dH_t$ holds for $H_t := \dot F_t \circ
\phi_t^{-1}-\iota_{X_s}\lambda$. 
\end{proof}

Now let $\psi^\lambda_s\in\Symp_0(A,\p A,\ostd)$ be the special
contraction from Proposition~\ref{Tspeccontr}.  
Let $\lambda_\std=\frac{-1}{2(1+r^2)\pi}d\theta$ be the standard
primitive of $\ostd$ (any other primitive would also do). 
Then for each $(s,\lambda)$ the $1$-form
$\alpha_s^\lambda:=(\psi^\lambda_s)^*\lambda_\std-\lambda_\std$ on $A$
is closed and vanishes near $\p A$. So by the relative Poincar\'e lemma, 
$$
   (\psi^\lambda_s)^*\lambda_\std-\lambda_\std = dF_s^\lambda
$$
for a unique smooth family of functions $F_s^\lambda$ that vanish near
the lower boundary component $\p_-A=\{a\}\times S^1$ of $A$. (We can
define $F_s^\lambda(w):=\int_{\gamma_w}\alpha_s^\lambda$ along any
path $\gamma_w$ from a base point on $\p_-A$ to $w$, which does not
depend on the path because every loop can be deformed into $\p_-A$
where $\alpha_s^\lambda$ vanishes.) Note that $F_s^\lambda$ will be
constant near the upper boundary component $\p_+A=\{b\}\times S^1$,
where the constant may depend on $s$ and $\lambda$.  

By Lemma~\ref{LformulaHtFt}, the family $F^\lambda_s$ is related to
a smooth family of Hamiltonians $\wt H^\lambda_s$ generating the isotopy  
$\psi^\lambda_s$ (for fixed $\lambda$) by the formula 
$$
   \wt H^\lambda_s = \frac{\p F^\lambda_s}{\p s}\circ
   (\psi^\lambda_s)^{-1}-\iota_{X^\lambda_s}\lambda_\std,
$$
where $\frac{d}{dt}\psi^\lambda_t=X^\lambda_t\circ
\psi^\lambda_t$. By construction, $\wt H_s^\lambda$ vanishes near the lower boundary
component $\p_-A$ of $A$ and it is constant near the upper boundary
component $\p_+A$ (where the constant may vary with $s$ and $\lambda$).
Further, since $\psi_s^\lambda$ is constant near its ends
in both $s$ and $\lambda$, we have $\wt H^\lambda_s=0$ for $|\lambda|\geq
\frac{\pi}{4}$ and for $s<2\epsilon$, $s>1-2\epsilon$ with some
$\epsilon>0$. 

Note that, since $\psi^0_s$ preserves the equator,
the Hamiltonian vector field $X^0_s$ is tangent to $E$ for all $s$. So
the restriction $\wt H^0_s|_E$ is constant for all $s$ and defines a
function $\wt H_E(s)$. For reasons that will become clear in the next
subsection, we wish to modify $\wt H$ to make this function vanish. For
this, we pick be a smooth cutoff function $\rho:\R\to[0,1]$ with 
$\rho(0)=1$ and support in $[-\frac{\pi}{4},\frac{\pi}{4}]$ and define
$$
   H_s^\lambda := \wt H_s^\lambda - \rho(\lambda)\wt H_E(s). 
$$ 
Since $H_s^\lambda$ differs from $\wt H_s^\lambda$ only by a function
of $s$ and $\lambda$, it still has the same Hamiltonian vector field
and thus still generates the family $\psi_s^\lambda$. By
construction, $H_s^\lambda$ depends only on $s$ and $\lambda$ near the
boundary $\p A$ (with possibly different functions at the two
boundary components), $H^\lambda_s=0$ for $|\lambda|\geq
\frac{\pi}{4}$ and for $s<2\epsilon$, $s>1-2\epsilon$, and 
$$
   H_s^0|_E=0\quad\text{for all }s. 
$$
We define in spherical coordinates on the base the squares
\begin{align*}
   Q &:= \left\lbrace (\mu,\lambda)\in S^2\setminus\left\lbrace
   N,S\right\rbrace \;\bigl|\; 2\epsilon \leq \mu \leq 1-2\epsilon,\ |\lambda|\leq
   \frac{\pi}{4}\right\rbrace,\cr  
   \wt{Q} &:= \left\lbrace (\mu,\lambda)\in S^2\setminus\left\lbrace
   N,S\right\rbrace \;\bigl|\; \epsilon \leq\mu \leq 1-\epsilon, \ |\lambda|\leq
   \frac{\pi}{3}\right\rbrace. 
\end{align*}
Note that $Q\subset\inn\wt Q$ and $\wt Q\subset\inn B$, where $B$ is
the region defined at the beginning of this section over which $\om$
is split. The family $H_s^\lambda$ constructed above gives rise to 
a smooth function 
$$
   H \colon \wt{Q}\times A \to \mathbb{R},\qquad (\lambda,\mu,w) \mapsto
   H^\lambda_\mu(w).
$$ 
Let us write the fibre sphere as 
$$
   S^2={\rm Cap}_{N}\cup A \cup {\rm Cap}_{S},
$$
where ${\rm Cap}_N$ and ${\rm Cap}_S$ denote the northern and southern
polar caps, respectively. Then we can extend $H$ first to
$\wt{Q}\times S^2$ by the corresponding functions of $(\lambda,\mu)$
on the southern and northern polar caps, and then to all of 
$S^2\times S^2$ by zero outside $\wt Q\times S^2$. We still denote the
resulting function by $H:S^2\times S^2\to\R$. By construction, $H$ has
support in $Q\times S^2$, it depends only on $(\lambda,\mu)$
outside $Q\times A$, and $H(0,\mu)|_E\equiv 0$ for all $\mu$, where
we denote $H(\lambda,\mu):=H|_{p_1^{-1}(\lambda,\mu)}$.

\subsection{A special symplectic connection}

Recall that we consider a relative symplectic fibration
$(\FF_\std,\omega,L_\std,S_0,S_\infty)$ such that the symplectic
form $\omega$ is split on the set $(B\times S^2)\cup(S^2\times(U_0\cup
U_\infty))$. Our current goal is to change the symplectic form
$\omega$, in its relative cohomology class in $H^2(S^2\times
S^2,L_\std; \mathbb{R})$, to a form $\om'$ which has trivial
holonomy around the circles of latitude. To explain the idea, consider
a circle of latitude $C^\lambda$ (cf.~Figure~\ref{f48}). 
As the symplectic form is split over $B$, its parallel transport
equals the identity along the part of $C^\lambda$ lying within $B$, so
the holonomy $\phi^\lambda$ is realised by 
travelling along the part of $C^\lambda$ outside $B$.
 
\begin{figure}[h!]
 \centering
 \psfrag{d}{$\phi^\lambda$ is realised here}
\psfrag{B}{$Q$}
\psfrag{C}{$Q\cap C^\lambda$; $\psi^\lambda$ shall be realised here}
\psfrag{q}{$Q$}
\psfrag{mu}{$\mu$}
\psfrag{z0}{$z_0$}
 \includegraphics[width=9cm]{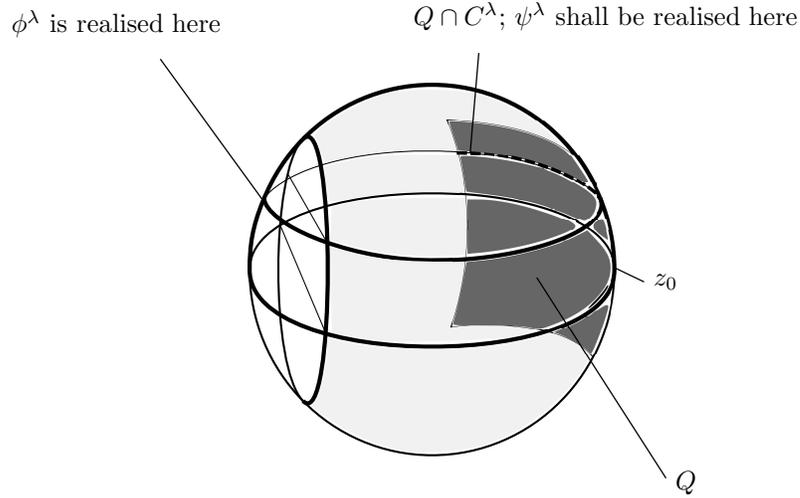}
 \caption{The path $Q\cap C^\lambda$ and $B\cap C^\lambda$}
 \label{f48}
\end{figure}

The idea is now to modify $\omega$ to $\omega'$ such that the
symplectic connection of $\omega'$ agrees with that of $\omega$
outside $Q\times S^2$ and realises the inverse holonomy $\psi^\lambda$
along $C^\lambda\cap Q$ for all $\lambda$. 

For the following computations, let us rename the coordinates
$(\lambda,\mu)$ to 
$$
   x:=\mu\in[0,1],\quad y:=\lambda\in[-\frac{\pi}{3},\frac{\pi}{3}]. 
$$ 
Recall that the function $H:S^2\times S^2\to\R$ constructed in the
previous subsection has support in $Q\times S^2$, where $Q=[2\epsilon,1-2\epsilon]\times
[-\frac{\pi}{4},\frac{\pi}{4}]$ in the new coordinates $(x,y)$. 
Consider the closed $2$-form 
$$
   \Omega_H=\omega+dH\wedge dx
$$ 
on $Q\times A$, extended by $\om$ to a form on all of $S^2\times 
S^2$. Since $\Omega_H$ is vertically nondegenerate, the 
$\Omega_H$-orthogonal complements to the tangent spaces of the fibres
of $p_1$ induce a symplectic connection on $S^2\times S^2$.

\begin{lemma}\label{lem:hol-latitude}
(a) The holonomy of $\Om_H$ along each circle of latitude $C^\lambda$ is
trivial.

(b) The closed form $\Om_H$ vanishes on $L_\std$ and is relatively
cohomologous to $\om$. 
\end{lemma}

\begin{proof}
(a) Recall that $H$ depends only on $x$ and $y$ outside the set
$Q\times A$, so $\Om_H$ and $\om$ differ there by the pullback
$dH\wedge dx=p_1^*\alpha$ of a $2$-form $\alpha$ from the base. Since
adding the pullback of a 
$2$-form from the base does not change the symplectic connection
(because $\iota_v(p_1^*\alpha)=0$ for every vertical vector $v$), the
induced connections of $\Om_H$ and $\omega$ agree outside the set
$Q\times A$. Within $Q\times A$ the form
$\omega=p_1^*\sigma_0+p_2^*\ostd$ is split, so that its induced
connection is flat. The horizontal spaces of the induced connection of
$\Omega_H$ are spanned by the horizontal lifts of the coordinate
vector fields $\p_x,\p_y$. These can be easily seen to be 
$$
   \wt\p_x = \p_x+X_{H_x^y},\qquad \wt\p_y = \p_y,
$$
where $X_{H_x^y}$ is the Hamiltonian vector field of the Hamiltonian
function $H_x^y(w)=H(x,y,w)$ on the annulus $(A,\ostd)$. 
To see this, let us write $\tilde\p_x=\p_x+v_x$ with a
vertical vector $v_x$. This is horizontal iff 
$$
   0 = \Omega_H(\wt\p_x,v) = \Omega_H(\p_x,v) + \Omega_H(v_x,v) =
   -dH(v) + \ostd(v_x,v) 
$$ 
for all vertical vectors $v$, which just means that $v_x$ is the
Hamiltonian vector field of $H_x^y$ with respect to $\ostd$. A
similar calculation shows that $v_y=0$. 

It follows that the parallel transport of $\Om_H$ along an interval of
latitude $C^\lambda\cap Q\cong [2\epsilon,1-2\epsilon]\times\{y\}$ is the time-$1$
map of the Hamiltonian flow of the time-dependent Hamiltonian
$H_s^\lambda$. By construction of $H_s^\lambda$, this is the inverse
$\psi^\lambda$ of the holonomy of $\om$, and thus of $\Om_H$, along
the interval $C^\lambda\setminus Q$. Hence the total holonomy of
$\Om_H$ along each circle of latitude $C^\lambda$ is trivial. 

(b) By construction, the horizontal vector field
$\wt\p_x=\p_x+X_{H_x^y}$ is tangent to $L_\std$. Let $v$ be the
vertical vector field along $L_\std$ given by the positively oriented
unit tangent vectors to the equators in the fibres. Since $\Om_H(\wt
\p_x,v)=0$ by definition of horizontality, this shows that $L_\std$ is 
Lagrangian for $\Om_H$. 
Finally, let us compute the relative homology class of $\Om_H$ in
$H^2(S^2\times S^2,L_\std)$. For this, we evaluate $\Om_H$ on the
generators of $H_2(S^2\times S^2,L_\std)$ in Lemma~\ref{lem:homology}: 
\begin{align*}
   \int_{S^2\times\pt}\Om_H &= \int_{\pt\times S^2}\Om_H = 1, \cr
   \int_{\pt\times D_\lh}\Om_H &= \int_{\pt\times D_\lh}\om = \frac{1}{2}, \cr
   \int_{D_\lh\times\pt}\Om_H &= \frac{1}{2} + \int_{(Q\cap\{y\leq
       0\})\times\{e\}}dH\wedge dx = \frac{1}{2} +
   \int_0^1H(x,0,e)dx = \frac{1}{2}.
\end{align*}
Here in the last equation $e\in E$ is a base point on the equator in
the fibre and we have used the normalisation condition $H(x,0,e) =
H(x,0)|_E\equiv 0$ from the previous subsection. Since $\om$ takes the
same values on these classes by monotonicity of $L_\std$, this shows
that the relative cohomology classes of $\Om_H$ and $\om$ agree. 
\end{proof}

%

Let us analyse when the form $\Omega_H$ is symplectic. Since it is
closed, this is equivalent to the form $\Omega_H\wedge \Omega_H$ being
a volume form on $S^2\times S^2$. This is clearly satisfied outside
the set $Q\times S^2$ because there $H\equiv 0$. On the set $Q\times
S^2$, the form $\om$ is split of the form 
$\omega=p_1^*\sigma_0+p_2^\ast \ostd$. We work in the chosen
coordinates $x,y$ and write the form on the base as
$$
   \sigma_0=f(x,y)dx\wedge dy
$$ 
with a positive function $f$. A short computation yields
$$
   \Omega_H\wedge \Omega_H=\left( 1-\frac{1}{f}\dd{H}{y}\right)
   \omega\wedge \omega.
$$ 
So $\Omega_H$ will be symplectic iff 
\begin{equation*}\label{eq:symp}
   1-\frac{1}{f}\dd{H}{y}>0
\end{equation*}
everywhere. A priori, this need not be true for the given function
$H$, but it can be remedied by the inflation procedure in the next
subsection.

\subsection{Inflation}

In this subsection we recall the inflation procedure of McDuff and
Lalonde~\cite{LalMcD}, suitably adapted to our situation. 
Let $f_\sigma, \bar{f}_\tau$ be two smooth nonnegative bump functions
on $S^2$, where we think of $f_\sigma$ as living on the fibre sphere  
and of $\bar{f}_\tau$ as living on the base sphere; see
Figure~\ref{fig:inf}. We require that 
$$
   {\rm supp}(f_\sigma)\subset (U_0\cup U_\infty)\setminus A = {\rm
     Cap}_S\amalg {\rm Cap}_N, 
$$
where $U_0,U_\infty$ are the neighbourhoods of $S,N$ over which $\om$ is
split and $A$ is the annulus from the previous subsection, and that
$$
   \int_{{\rm Cap}_S}f_\sigma \ostd=\int_{{\rm Cap}_N}f_\sigma\ostd=\frac{1}{2}.
$$ 
In particular, $\int_{S^2}f_\sigma\ostd=1$. The function $\bar{f}_\tau$
is required to have support in $\wt{Q}$ and satisfy 
$$
   \bar{f}_\tau(x,y)=\bar{f}_\tau(x,-y)
$$ as well as $\bar{f}_\tau|_Q\equiv 1$. We define
$$
   f_\tau:=\frac{\bar{f}_\tau}{af}\quad\text{with}\quad
   a:=\int_{\wt{Q}}\frac{\bar{f}_\tau}{f}\sigma_0 =
   \int_{\wt{Q}}\bar{f}_\tau dx\wedge dy,
$$
where $\sigma_0 = f(x,y)dx\wedge dy$ as above. Then 
$\int_{\wt{Q}}f_\tau\sigma_0=\frac{1}{a}\int_{\wt{Q}}\bar{f}_\tau
dx\wedge dy=1$, which by the symmetry of $\bar{f}_\tau$ implies
\begin{equation*}
   \int_{\wt{Q}\cap \left\lbrace y\geq
     0\right\rbrace}f_\tau\sigma_0 =
   \frac{1}{a}\int_{\wt{Q}\cap\left\lbrace y\geq 0\right\rbrace
    }\bar{f}_\tau dx\wedge dy=\frac{1}{2}.  
\end{equation*}

 \begin{figure}[h!]
 \centering\psfrag{DN}{${\rm Cap}_N$}
\psfrag{DS}{${\rm Cap}_S$}
\psfrag{Dz0}{$\wt{Q}$}
\psfrag{z0}{$z_0$}
 \psfrag{N}{$N$}
 \psfrag{S}{$S$}
 \psfrag{Base}{On the base}
 \psfrag{Bases}{On the fibre}
 \psfrag{fsigma}{$f_\sigma$}
 \psfrag{fsigmat}{$\bar{f}_\tau$}
 \psfrag{sym}{$symmetric$}
 
 \includegraphics[width=11cm]{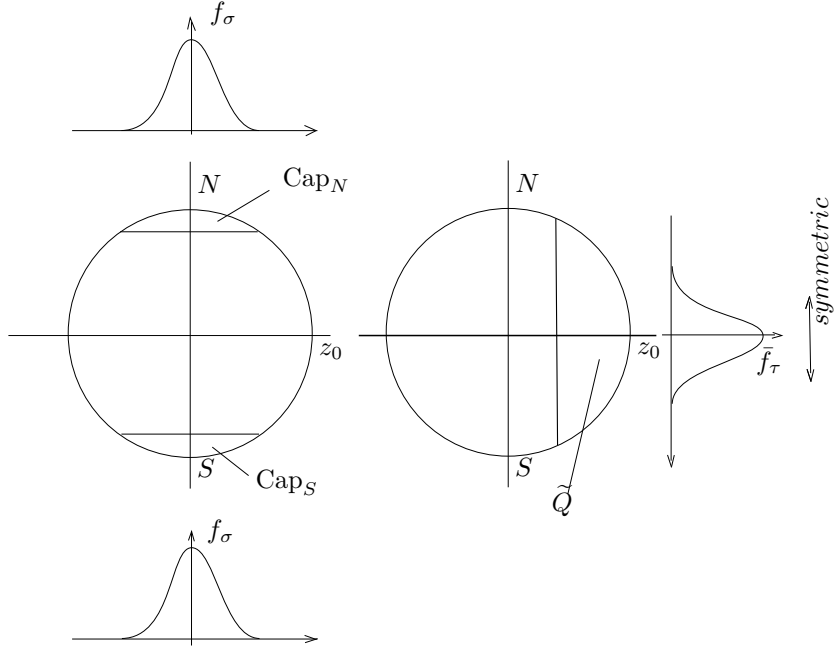}
 \caption{The functions $f_\sigma$ and $\bar{f}_\tau$}
 \label{fig:inf}
\end{figure}

We define the two non-negative $2$-forms 
$$
   \sigma:=f_\sigma \ostd,\qquad \tau:=f_\tau \sigma_0
$$ 
on $S^2$ and consider the family of $2$-forms on $S^2\times S^2$
\begin{equation}\label{eq:omc}
   \omega_c := \frac{1}{c+1}\left( \omega+cp_1^\ast \tau +cp_2^\ast
   \sigma \right),\qquad c\geq0.
\end{equation}

\begin{lemma}\label{lem:omc}
For each $c\geq 0$ the form $\om_c$ has the following properties:

(a) $\om_c$ is symplectic and $L_\std$ is Lagrangian for $\om_c$;

(b) $\om_c$ is cohomologous to $\om_0=\om$ in $H^2(S^2\times
S^2,L_\std;\R)$; 

(c) $\om_c$ induces the same symplectic connection as $\om$.
\end{lemma}

\begin{proof}
(a) First note that $\omega_c$ is closed for all $c\geq 0$ and 
$$
   \omega_c=\frac{1}{c+1}\omega\quad\text{outside }W:=W_0\cup W_\infty,
$$
where $W_0,W_\infty$ are the sets from~\eqref{eq:split2} on which $\om$
is split. On the set $W_0$, 
\begin{equation}\label{eq:omc1}
   \omega_c=\frac{1}{c+1}\Bigl((1+cp_1^\ast f_\tau)p_1^\ast
   \sigma_0+(1+cp_2^\ast f_\sigma)p_2^\ast \ostd\Bigr),
\end{equation}
and therefore
\begin{equation}\label{eq:omc2}
   \omega_c\wedge \omega_c=\frac{1}{(c+1)^2}(1+c p_1^\ast f_\tau)(1+c
   p_2^\ast f_\sigma)\,\omega \wedge \omega>0
\end{equation}
because $c, f_\tau, f_\sigma$ are nonnegative.
For the set $W_\infty$, we write $\sigma_\infty=g\sigma_0$
for a positive function $g$. Then on $W_\infty$ we have
$$
   \omega_c=\frac{1}{c+1}\Bigl((p_1^\ast g+cp_1^\ast f_\tau)p_1^\ast
   \sigma_0+(1+cp_2^\ast f_\sigma)p_2^\ast \ostd\Bigr),
$$ 
and again positivity of $g$ and nonnegativity of $c,f_\tau,f_\sigma$
implies 
$$
   \omega_c\wedge \omega_c=\frac{1}{(c+1)^2}(p_1^\ast g+c p_1^\ast
   f_\tau)(1+c p_2^\ast f_\sigma)\,\omega \wedge \omega>0.
$$
This proves that $\om_c$ is symplectic. The torus $L_\std$ is
Lagrangian for $\om_c$ because all pullback forms from the base or
the fibre vanish on $L_\std$. 

(b) To show that $\omega_c$ is relatively cohomologous to $\omega$,
we evaluate it on the basis of $H_2(S^2\times S^2,L_\std)$ from
Lemma~\ref{lem:homology}. Using $\int_{S^2}\sigma=\int_{S^2}\tau=1$, 
$\int_{D_\lh}\sigma=\int_{{\rm Cap}_S}\sigma=1/2$ and
$\int_{D_\lh}\tau=\int_{\wt Q\cap\{y\leq 0\}}\tau=1/2$, we compute
with the point $z_0\in E$ on the equator in the base or fibre:
\begin{align*}
   \int_{\pt\times S^2}\om_c &= \frac{1}{c+1} \int_{\pt\times S^2}
   (\omega+cp_2^\ast \sigma) = \frac{1}{c+1}(1+c)=1, \cr 
   \int_{S^2\times\pt}\om_c &= \frac{1}{c+1}\int_{S^2\times\pt}
   (\omega+c p_1^\ast \tau) = \frac{1}{c+1}(1+c)=1, \cr 
   \int_{\pt\times D_\lh}\om_c &= \frac{1}{c+1}\int_{\{z_0\}\times
     D_\lh} (\omega+cp_2^\ast
   \sigma)=\frac{1}{c+1}\Bigl(\frac{1}{2}+\frac{c}{2}\Bigr)=\frac{1}{2}, \cr 
   \int_{D_\lh\times\pt}\om_c &= \frac{1}{c+1}\int_{D_\lh\times\{z_0\}} 
   (\omega+c p_1^\ast \tau)=\frac{1}{c+1}\Bigl(\frac{1}{2}+\frac{c}{2}\Bigr) =
   \frac{1}{2}.
\end{align*}
By monotonicity of $L_\std$, the form $\om$ takes the same values on
these classes, so $[\om_c]=[\om]\in H^2(S^2\times S^2,L_\std;\R)$. 

(c) On the set $W=W_0\cup W_\infty$ the forms $\omega_c$ and $\om$ are
both split, hence both symplectic connections are flat and the
horizontal subspaces are the tangent spaces to the other cartesian 
factor. Outside $W$ we have $\omega_c=\frac{1}{c+1}\omega$ and, since
the symplectic complements to the fibres are not affected by scaling
of the symplectic form, the symplectic connections of
$\omega_c$ and $\om$ agree here as well.
\end{proof}


{\bf A new symplectic connection. }
Now recall that the function $H$ from the previous section is a
pullback from the base outside the set $S^2\times A$. On the set
$S^2\times A$, the function $p_2^\ast f_\sigma$ vanishes and thus
$\omega_c=\frac{1}{c+1}\Bigl((1+p_1^*f_\tau)p_1^\ast
\sigma_0+p_2^\ast \ostd\Bigr)$.  
In particular, on this set the restriction of $\omega_c$ to the fibres
it is just the standard form $\ostd$ scaled by $\frac{1}{c+1}$. Now
the fibrewise Hamiltonian vector field of the rescaled function
$\frac{1}{c+1}H$ with respect to $\frac{1}{c+1}\ostd$ equals the
fibrewise Hamiltonian vector field $X_{H_x^y}$ of $H$ with respect
to $\ostd$. So the horizontal lift of $\p_x$ with respect to the
closed $2$-form
\begin{equation}\label{eq:Omc}
   \Om_H^c:=\om_c+\frac{1}{c+1}dH\wedge dx
\end{equation}
agrees with its horizontal lift $\p_x+X_{H_x^y}$ with respect to
$\Om_H$ (see the proof of Lemma~\ref{lem:hol-latitude}), and since the
horizontal lift of $\p_y$ is $\p_y$ in both cases, we see that
$\Om_H^c$ and $\Om_H$ define the same symplectic connection for all
$c\geq 0$. Moreover, the proof of Lemma~\ref{lem:hol-latitude}(b) shows
that $\Om_H^c$ vanishes on $L_\std$ and is relatively cohomologous to
$\om_c$, and thus to $\om$ by Lemma~\ref{lem:omc}. 

{\bf Symplecticity. }
Again, let us analyse when the form $\Om_H^c$ is symplectic. 
Outside $Q\times S^2$, the form $\Omega^c_H$ is just $\omega_c$, which
is symplectic by Lemma~\ref{lem:omc}. On the set $Q\times S^2 \subset
W_0$, using equations~\eqref{eq:omc1} and~\eqref{eq:omc2} we compute
\begin{align*}
   \omega_c \wedge \frac{1}{c+1}dH\wedge dx
   &= \frac{1}{(c+1)^2}\left(-\dd{H}{y}\right)\left(1+c
   p_2^\ast f_\sigma \right)dx\wedge dy\wedge p_2^*\ostd \cr
   &= \frac{1}{2f(c+1)^2}\left(-\dd{H}{y}\right)\left(1+c
   p_2^\ast f_\sigma \right)\omega\wedge \omega, \cr
  \Omega^c_H\wedge \Omega^c_H 
   &= \omega_c\wedge \omega_c+2\omega_c \wedge \frac{1}{c+1}dH\wedge
   dx \cr
   &= \frac{1}{(c+1)^2}\left(1+cp_1^\ast f_\tau - \frac{1}{f}\dd{H}{y}
   \right)\left(1+cp_2^\ast f_\sigma \right) \omega\wedge \omega.
\end{align*} 
Now $1+cp_2^\ast f_\sigma\geq 1$ for all $c$ by nonnegativity of $c$
and $f_\sigma$. Moreover, by the choice of $f_\tau$ we have $p_1^\ast
f_\tau=\frac{1}{af}$ on $Q\times S^2$. Hence $\Om_H^c$ is symplectic
iff
\begin{equation*}
   1+\frac{1}{f}\left(\frac{c}{a}-\dd{H}{y}\right)>0
\end{equation*}
on $Q\times S^2$. But this is satisfied for 
\begin{equation}\label{eq:c}
   c\geq C:=a\,\max_{Q\times S^2}\left|\frac{\p H}{\p y}\right|.
\end{equation} 
We summarize the preceding discussion in

\begin{lemma}\label{lem:OmHc}
The closed $2$-form $\Om_H^c$ vanishes on $L_\std$, is relatively
cohomologous to $\om_c$ (and thus $\om$), and has trivial holonomy
along each circle of latitude $C^\lambda$ for each $c\geq 0$. 
Moreover, $\Om_H^c$ is symplectic for $c\geq C$ given
by~\eqref{eq:c}. 
\end{lemma}

\subsection{Killing the holonomy along circles of latitude}

We denote the $0$-meridian by $m_0:=\left\lbrace (\lambda,\mu) \in
S^2\mid \mu=0 \right\rbrace$ in spherical coordinates. Putting the
previous subsections together, we can now prove 

\begin{proposition}\label{Pkillmonolatitude}
Let $\SS=(\FF_\std,\omega,L_\std,S_0,S_\infty)$ be a relative
symplectic fibration such that $\om$ is split 
on a neighbourhood of the fibre $F$ and the sections
$S_0,S_\infty$. Then there exists a homotopy of relative symplectic
fibrations $\SS_t=(\FF_\std,\omega_t,L_\std,S_0,S_\infty)$ with
$\SS_0=\SS$ such that the holonomy of $\SS_1$ along the circles of
latitude $C^\lambda$ is the identity for all $\lambda$. Moreover,
$\omega_1$ is split near the set $(m_0\times S^2)\cup S_0\cup S_\infty$.  
\end{proposition}

\begin{proof}
As explained at the beginning of Section~\ref{ss:setup}, we may assume
that $\om$ is split on a set $(B\times S^2)\cup\bigl((S^2\times
(U_0\cup U_\infty)\bigr)$, where the ball $B\subset S^2$ contains the
$0$-meridian $m_0$. 
Let $H$ be the Hamiltonian function constructed in
Section~\ref{ConstrH} and let $C$ be the constant defined
in~\eqref{eq:c}. For $c\in[0,C]$, let $\om_c$ be the form defined
in~\eqref{eq:omc}. By Lemma~\ref{lem:omc},
$(\FF_\std,\omega_c,L_\std,S_0,S_\infty)$ gives a homotopy of
relative symplectic fibrations from $\SS$ to
$(\FF_\std,\omega_C,L_\std,S_0,S_\infty)$. For $t\in[0,1]$,
consider the forms 
$$
   \Om_{tH}^C = \om_C+\frac{t}{C+1}dH\wedge dx = (1-t)\om_C+t\,\Om_H^C
$$
as in~\eqref{eq:Omc} (with $H$ replaced by $tH$ and $c$ by $C$). By
Lemma~\ref{lem:OmHc} (applied to $tH$),
$(\FF_\std,\Om_{tH}^C,L_\std,S_0,S_\infty)$ gives a homotopy of 
relative symplectic fibrations from
$(\FF_\std,\omega_C,L_\std,S_0,S_\infty)$ to
$\SS_1=(\FF_\std,\Om_H^C,L_\std,S_0,S_\infty)$. By the same
lemma, $\SS_1$ has trivial holonomy along all circles of latitude. 
Hence the concatenation of the previous two homotopies gives the
desired homotopy $\SS_t$. For the last assertion, simply observe that 
by construction all symplectic forms in this homotopy agree with the
original split form $\om$ near $(m_0\times S^2)\cup S_0\cup S_\infty$.  
\end{proof}

\begin{remark}\label{rem:Ham}
The point of departure for the preceding subsections was the
standardisation provided by
Proposition~\ref{prop:section-standard}. If rather than making the
symplectic form $\om$ split we had made it equal to $\om_\std$
near $(m_0\times S^2)\cup S_0\cup S_\infty$ (as suggested in
Remark~\ref{rem:section-standard}), then the holonomies
$\phi^\lambda$ would lie in the subgroup $\Ham(A,\p
A,\ostd)\subset\Symp_0(A,\p A,\ostd)$ of symplectomorphisms generated
by Hamiltonians with compact support in $A\setminus\p A$ and the whole
construction could be performed in that subgroup (which is also
contractible). However, since we change the normalisation of the
Hamiltonians $H^\lambda_s$ anyway to make them vanish on the equator, 
we would gain nothing from working in this subgroup. 
\end{remark}

\subsection{Killing all the holonomy}

Now we will further deform the relative symplectic fibration from the
previous subsection to one which has trivial holonomy along {\em all}
closed curves in the base. We begin with a simple lemma.

\begin{lemma}\label{Llinearformsinterpol}
Let $\omega,\omega'$ be linear symplectic forms on $\mathbb{R}^4$
which define the same orientation and agree on a real codimension one
hyperplane $H$. Then $\omega_t:=(1-t)\omega+t\omega'$ is symplectic for all
$t\in[0,1]$. 
\end{lemma}

\begin{proof}
Take a symplectic basis $e_1,f_1,e_2,f_2$ for $\omega$ such that
$e_1,f_1,e_2$ is a basis of $H$. Take a vector
$f_2'=a_1e_1+b_1f_1+a_2e_2+b_2f_2$ such that $e_1,f_1,e_2,f_2'$ is a
symplectic basis for $\omega'$. Since $\omega,\omega'$ induce the
same orientation, we have $b_2>0$, and therefore 
$$
   \omega(e_2,f_2')=b_2>0,\qquad \omega'(e_2,f_2)=\frac{1}{b_2}>0.
$$
For $\omega_t:=(1-t)\omega+t\omega'$ we find 
$$
   \omega_t\wedge \omega_t=(1-t)^2\omega\wedge\om+2t(1-t)\omega\wedge
   \omega'+t^2\omega'\wedge\om',
$$
and therefore 
\begin{align*}
   \omega_t\wedge\omega_t(e_1,f_1,e_2,f_2')
   &= 2(1-t)^2\omega(e_1,f_1)\omega(e_2,f_2')+2t^2\omega'(e_1,f_1)\omega'(e_2,f_2')\cr
   &\ \ \ +2t(1-t)\Bigl(\omega(e_1,f_1)\omega'(e_2,f_2')+\omega(e_2,f_2')\omega'(e_1,f_1)
   \Bigl) \cr 
   &> 0.  
\end{align*}
\end{proof}

Recall the definition of the $0$-meridian $m_0$ from the previous
subsection. 

\begin{proposition}\label{Ptrivholonomy}
Let $\SS=(\FF_\std,\om,L_\std,S_0,S_\infty)$ be a relative
symplectic fibration which is split near the set $S_0\cup
S_\infty\cup(m_0\times S^2)$
and has trivial holonomy around all circles of latitude $C^\lambda$.
Then there exists a homotopy of relative symplectic fibrations
$\SS_t=(\FF_\std,\om_t,L_\std,S_0,S_\infty)$ with $\om_0=\om$
and $\om_1=\om_\std$. 
\end{proposition}

\begin{proof}
The idea of the proof is to use parallel transport along 
circles of latitude to define a fibre preserving diffeomorphism
$\phi$ of $S^2\times S^2$ which pulls back the symplectic form $\om$
to a form which agrees with the standard form $\omega_\std$ on 
$C^\lambda\times S^2$ for all $\lambda$, and then apply
Lemma~\ref{Llinearformsinterpol}. 

\begin{figure}[h!]
 \centering
\psfrag{Px}{$\id$}
\psfrag{plx}{$P^\lambda_\mu$}
\psfrag{fN}{$p_1^{-1}(N)$}
\psfrag{fxy}{$p_1^{-1}(\lambda,\mu)$}
\psfrag{fx0}{$p_1^{-1}(\lambda,0)$}
\psfrag{m0}{$m_0$}
\psfrag{p1}{$\id$}
\psfrag{p2}{$\id$}
\psfrag{x0}{$(\lambda,0)$}
\psfrag{xy}{$(\lambda,\mu)$}
\psfrag{alpha}{$\id$}
\psfrag{phi}{$\phi$}

 \includegraphics[width=12cm]{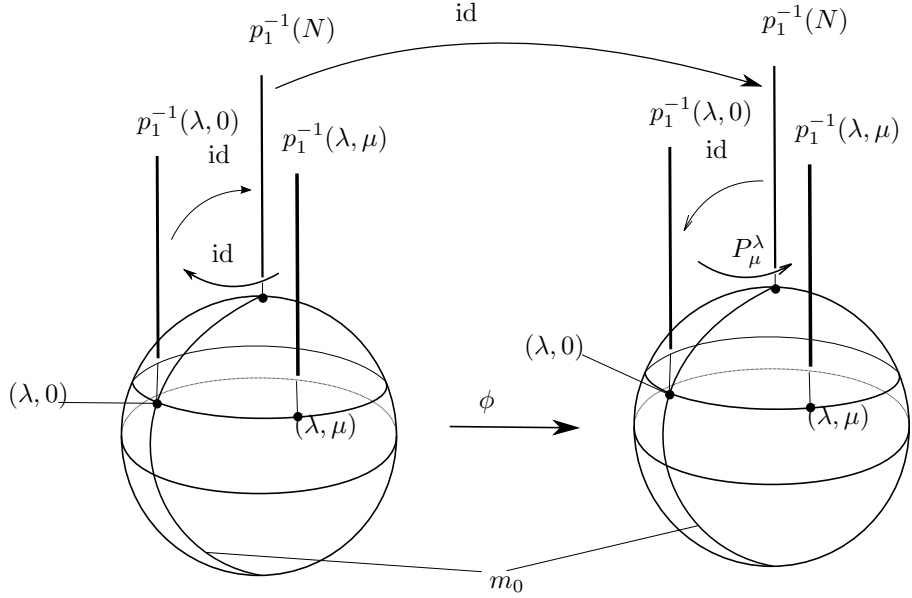}
 \caption{The maps $P^\lambda_\mu$ and the construction of $\phi$}
 \label{f410}
\end{figure}

For each $\lambda,\mu$ let
$$
   P^\lambda_\mu \colon \left\lbrace (\lambda,0) \right\rbrace\times
   S^2  \to  \left\lbrace (\lambda,\mu) \right\rbrace\times S^2
$$ 
be the parallel transport of (the symplectic connection on $p_1$
defined by) $\om$ along the circle of latitude
$C^\lambda$ from $(\lambda,0)$ to $(\lambda,\mu)$. Since $\om$ has
trivial holonomy along $C^\lambda$, this does not depend on the path
in $C^\lambda$ and is thus well-defined. Note that, due to the fact
that $\om$ is split near $S_0\cup S_\infty$, the map $P^\lambda_\mu$
equals the identity near the north and south poles $N,S$ in the fibre. 

We define a fibre preserving diffeomorpism $\phi$ of $S^2\times S^2$ by
parallel transport on the left sphere in 
Figure~\ref{f410} with respect to the standard form $\omega_\std$ 
first going backwards along the circle of latitude until we hit the
meridian $m_0$, then upwards along $m_0$ until we hit the north
pole $N$, then by the identity 
to the fibre over the north pole of the right sphere, then by
parallel transport with respect to $\om$ along $m_0$, and finally along
the circle of latitude to land in the fibre over the original point
$(\lambda,\mu)$. Since parallel transport with respect to the
symplectic connection $\omega_\std$ is the identity for all paths, as is
parallel transport with respect to $\om$ along paths in $m_0$ (since
$\om$ is split over $m_0$), we will
not explicitly include these maps in the notation. Then we can
write the preceding definition in formulas as 
$$
   \phi \colon S^2\times S^2\to S^2\times S^2,\qquad 
   \Bigl((\lambda,\mu),w\Bigr) \mapsto
   \Bigl((\lambda,\mu)\,,\,P^\lambda_\mu(w)\Bigr).  
$$
Note that $\phi$ is smooth for $(\lambda,\mu)$ near the north and
south poles because there $P^\lambda_\mu=\id$. For $z\in S^2$ let us
denote by $\om_z$ the restriction of $\om$ to the fibre
$F_z=\{z\}\times S^2$. We claim that $\phi$ has the following properties:
\begin{itemize}
\item[(a)] $\phi$ restricts to symplectomorphisms
  $(F_z,\ostd)\to(F_z,\om_z)$ on all fibres; 
\item[(b)] $\phi$ preserves the Clifford torus $L_\std$;
\item[(c)] $\phi$ equals the identity near $S_0\cup S_\infty\cup(m_0\times
   S^2)$; 
\item[(d)] $\phi$ is isotopic to the identity through fibre preserving
  diffeomorphisms $\phi_t$ that preserve $L_\std$ and equal the
  identity near $S_0\cup S_\infty\cup(m_0\times S^2)$.   
\end{itemize}

For (a), note that $\om$ restricts to
$\ostd$ on the fibre $F_N$ over the north pole (because $\om$ is
split there), so the identity defines a symplectomorphism
$(F_N,\ostd)\to (F_N,\om_N)$. Now (a) follows because
parallel transport is symplectic. 

Property (b) follows from the fact that $L_\std$ is given by parallel
transport of the equator in the fibre around the equator in the base,
hence $P^0_\mu(E)=E$ and thus
$$
   \phi(\left\lbrace(0,\mu) \right\rbrace\times E) =
   \left\lbrace (0,\mu) \right \rbrace \times P^0_\mu(E) = 
   \left\lbrace (0,\mu) \right \rbrace\times E.
$$
Property (c) holds because $\om$ is split near $S_\infty\cup S_0\cup(m_0\times
   S^2)$. 

For (d), consider the map
$$
   P\colon
   R:=[-\frac{\pi}{2},\frac{\pi}{2}]\times[0,2\pi]\to\Diff(A,\p
   A),\qquad (\lambda,\mu)\mapsto P^\lambda_\mu. 
$$
It maps the boundary $\p R$ to $\id$ and the interval
$\{0\}\times[0,2\pi]$ to the subspace $\Diff(A,\p A;E)\subset
\Diff(A,\p A)$ of diffeomorphisms preserving the equator $E$ (as a
set, not pointwise). By Corollary~\ref{cor:Diff}, the loop
$[0,2\pi]\ni\mu\mapsto P^0_\mu$ is contractible in $\Diff(A,\p A;E)$. 
Using this and the fact (from Corollary~\ref{TDiff}) that
$\pi_2\Diff(A,\p A)=0$, we find a contraction of $P$ through maps
$P_t:R\to \Diff(A,\p A)$ sending $\p R$ to $\id$ and
$\{0\}\times[0,2\pi]$ to $\Diff(A,\p A;E)$. 
(The argument is analogous to the proof of
Proposition~\ref{Tspeccontr}, with the $1$-parametric family
$\psi^\lambda$ replaced by the $2$-parametric family $P^\lambda_\mu$.)
Then $\phi_t((\lambda,\mu),w):=((\lambda,\mu),P_t(\lambda,\mu)(w))$ is
the desired isotopy and the claim is proved. 

Now we construct the homotopy from $\om$ to $\om_\std$ in two steps. For the
first step, let $\phi_t$ be the isotopy in (d) from $\phi_0=\id$ to
$\phi_1=\phi$. Then
$\phi_t^{-1}(\SS)=(\FF_\std,\phi_t^*\om,L_\std,S_0,S_\infty)$ is
a homotopy of relative symplectic fibrations from $\SS$ to
$(\FF_\std,\phi^*\om,L_\std,S_0,S_\infty)$. 

For the second step, note that $\phi^*\om$ restricts to $\ostd$ on every
fibre by property (a). Moreover, since $\phi$ commutes with parallel transport along
$C^\lambda$ (with respect to $\om_\std$ and $\om$), the horizontal lifts
of vectors tangent to circles of latitude with respect to $\omega_\std$
and $\phi^*\om$ agree. Accordingly, $\omega_\std$ and $\phi^*\om$
agree on the 3-dimensional subspaces
$T_{((\lambda,\mu),w)}(C^\lambda\times S^2)$ in
$T_{((\lambda,\mu),w)}(S^2\times S^2)$ for all $((\lambda,\mu),w) \in
S^2\times S^2$. Thus, by Lemma \ref{Llinearformsinterpol}, the linear
interpolations $\om_t:=(1-t)\phi^*\om+t\omega_\std$ are symplectic for
all $t\in[0,1]$. 

We claim that 
$(\FF_\std,\om_t,L_\std,S_0,S_\infty)$ is a relative symplectic
fibration for all $t\in[0,1]$. For this, first note that $L_\std$ is
monotone Lagrangian for both $\om_\std$ and $\phi^*\omega$: For $\om_\std$
this is clear, and for $\phi^*\omega$ it follows from
$\phi(L_\std)=L_\std$ and monotonicity of $L_\std$ for
$\om$. Hence $L_\std$ is monotone Lagrangian for $\om_t$ for all
$t$. Next, since $\phi$ preserves fibres as well as the
sections $S_0,S_\infty$, the form $\phi^*\om$ and thus also the form
$\om_t$ is cohomologous to $\om_\std$ for all $t$. Finally, since $\phi$
preserves the sections $S_0,S_\infty$, they are symplectic for 
$\phi^*\om$ as well as $\om_\std$, hence for all $\om_t$. 

The desired homotopy $\om_t$ is the concatenation of the homotopies
constructed in the two steps. 
This concludes the proof of Proposition~\ref{Ptrivholonomy}.
\end{proof}

\begin{remark}\label{rem:trivholonomy}
The two steps in the proof of Proposition~\ref{Ptrivholonomy} could
have been performed in the opposite order: First homotope $\om$ to the
form $\phi_*\om_\std$ which has trivial holonomy along all closed curves
in the base, and then homotope $\phi_*\om_\std$ to $\om_\std$. The latter is
then a special case of the more general fact that two symplectic
fibrations with conjugate holonomy (e.g., both having trivial holonomy)
are diffeomorphic.
\end{remark}

\subsection{Proof of the main theorem and some consequences}\label{ss:proof}

We summarize the results of this and the previous section in

\begin{thm}[Classification of relative symplectic fibrations]\label{thm:fib}
Every relative symplectic fibration $\SS=(\FF,\om,L,\Sigma,\Sigma')$
on $S^2\times S^2$ is deformation equivalent to
$\SS_\std=(\FF_\std,\om_\std,L_\std,S_0,S_\infty)$.  
\end{thm}

\begin{proof}
By Proposition~\ref{prop:fix-F}, $\SS$ is diffeomorphic to a relative 
symplectic fibration of the form
$\wt\SS=(\FF_\std,\wt\omega,L_\std,S_0,S_\infty)$ for some symplectic
form $\wt\om$. Combining Proposition~\ref{prop:section-standard},
Proposition~\ref{Pkillmonolatitude} and
Proposition~\ref{Ptrivholonomy}, we find a homotopy from $\wt\SS$ to
$\SS_\std$.  
\end{proof}

The main theorem will be a consequence of Theorem~\ref{thm:fib} and the
following theorem of Gromov. 

\begin{theorem}[Gromov~\cite{Gro}]\label{G85}
Let $\phi \in Symp(S^2\times S^2,\omega_\std)$ act trivially on
homology. Then there exists a symplectic isotopy $\phi_t \in
Symp(S^2\times S^2,\omega_\std)$ with $\phi_0=\id$ and $\phi_1=\phi$. 
\end{theorem}

A first consequence is

\begin{cor}\label{cor:fib}
Let $\SS=(\FF,\om_\std,L,\Sigma,\Sigma')$ be a relative
symplectic fibration of $M=S^2\times S^2$, where $\om_\std$ is the
standard symplectic form. Then there exists a homotopy of relative
symplectic fibrations $\SS_t=(\FF_t,\om_\std,L_t,\Sigma_t,\Sigma_t')$ with
fixed symplectic form $\om_\std$ such that $\SS_0=\SS$ and
$\SS_1=(\FF_\std,\om_\std,L_\std,S_0,S_\infty)$. 
\end{cor}

\begin{proof}
By Theorem~\ref{thm:fib}, there exists a homotopy of relative
symplectic fibrations $\SS_t=(\FF_t,\om_t,L_t,\Sigma_t,\Sigma_t')$
with $\SS_1=(\FF_\std,\om_\std,L_\std,S_0,S_\infty)$ and  a
diffeomorphism $\phi$ of $S^2\times S^2$ acting trivially on homology
such that $\phi(\SS)=\SS_0$. After applying
Proposition~\ref{prop:fix-om-par} and modifying $\phi$ accordingly
(keeping the same notation), we may assume that $\om_t=\om_\std$ for
all $t\in[0,1]$. Then $\phi$ is a symplectomorphism with respect to
$\om_\std$, so by Gromov's Theorem~\ref{G85} it can be connected to the
identity by a family of symplectomorphisms $\phi_t$. Now the
concatenation of the homotopies $\phi_t(\SS)_{t\in[0,1]}$, and
$(\SS_t)_{t\in[0,1]}$ is the desired homotopy with fixed symplectic
form $\om_\std$. 
\end{proof}

\begin{proof}
[{\bf Proof of the Main Theorem~\ref{thm:main}}]  
The hypotheses of Theorem~\ref{thm:main} just mean that
$\SS=(\FF,\om_\std,L,\Sigma,\Sigma')$ is a relative symplectic
fibration. By Corollary~\ref{cor:fib}, $\SS$ can be connected to
$(\FF_\std,\om_\std,L_\std,S_0,S_\infty)$ by a homotopy of relative
symplectic fibrations $\SS_t=(\FF_t,\om_\std,L_t,\Sigma_t,\Sigma_t')$
with fixed symplectic form $\om_\std$. In particular, $L_t$ is an
isotopy of monotone Lagrangian tori (with respect to $\om_\std$)
from $L_0=L$ to $L_1=L_\std$. By Banyaga's isotopy extension
theorem, there exists a symplectic isotopy $\phi_t$ with $\phi_0=\id$
and $\phi_t(L)=L_t$ for all $t$ (see the proof of
Proposition~\ref{prop:fix-om-par} for the argument, ignoring the
symplectic sections). Since $M$ is simply connected, the symplectic
isotopy $\phi_t$ is actually Hamiltonian. 
\end{proof}

Another consequence of Theorem~\ref{thm:fib} is the following 
result concerning standardisation by diffeomorphisms. 

\begin{cor}
[Fixing the symplectic form]\label{cor:fix-om}
Let $\SS=(\mathcal{F},\omega,L,\Sigma,\Sigma')$ be a relative
symplectic fibration of $M=S^2\times S^2$. Then there exists a
diffeomorphism $\phi$ of $S^2\times S^2$ acting trivially on homology 
such that $\phi^{-1}(\SS)=(\wt\FF,\omega_\std,L_\std,S_0,S_\infty)$ for
some foliation $\wt\FF$. 
\end{cor}

\begin{proof}
By Theorem~\ref{thm:fib}, there exists a homotopy of relative
symplectic fibrations $\SS_t=(\FF_t,\om_t,L_t,\Sigma_t,\Sigma_t')$
with $\SS_1=(\FF_\std,\om_\std,L_\std,S_0,S_\infty)$ and a
diffeomorphism $\phi$ of $S^2\times S^2$ acting trivially on homology
such that $\phi(\SS)=\SS_0$. After applying
Proposition~\ref{prop:fix-om-par} and modifying $\phi$ accordingly
(keeping the same notation), we may assume that
$(\om_t,L_t,\Sigma_t,\Sigma_t')=(\om_\std,L_\std,S_0,S_\infty)$ for 
all $t\in[0,1]$. Then $\phi$ maps $\SS$ to
$(\FF_0,\om_\std,L_\std,S_0,S_\infty)$. 
\end{proof}

In particular, Corollary~\ref{cor:fix-om} implies that
every symplectic form $\om$ on $S^2\times S^2$ which is compatible
with a relative symplectic fibration can be pulled back to $\om_\std$
by a diffeomorphism $\phi\in\Diff_\id(M)$. This is a special case of
the deep result by Lalonde and McDuff~\cite{LalMcD} that every
symplectic form $\om$ on $S^2\times S^2$ which is cohomologous to
the standard form $\om_\std$ can be pulled back to $\om_\std$ by a
diffeomorphism $\phi\in\Diff_\id(M)$. In fact, the hard part of the
proof in~\cite{LalMcD} (using Taubes' correspondence between
Seiberg-Witten and Gromov invariants) consists in showing that 
any such $\om$ is compatible with a symplectic fibration with a
section.

\appendix

\section{Homotopy groups of some diffeomorphism groups}\label{app:diff}

In this appendix we collect some well-known facts about the diffeomorphism
and symplectomorphism groups of the disk and annulus. We fix numbers
$0<a<b$ and let 
$$
   D:=\left\lbrace z\in \mathbb{C}\;\bigl|\; |z|\leq b \right\rbrace,\qquad 
   A:=\left\lbrace z \in \mathbb{C}\;\bigl|\; a \leq |z|\leq b \right\rbrace 
$$
be equipped with the standard symplectic form
$\ostd=\frac{r}{\pi(1+r^2)^2}dr\wedge d\theta$ in polar coordinates on
$\C$ (the precise choice of the symplectic form does not matter
because they are all isomorphic up to scaling by Moser's theorem). 
We define the following diffeomorphism groups, all equipped with the
$C^\infty$ topology:
\begin{itemize}
\item $\Diff(D,\partial D)$ the group of diffeomorphisms of the closed
  disk $D$ that are equal to the identity in some neighbourhood of the boundary;
\item $\Diff(A,\partial A)$ the group of diffeomorphisms of the closed
  annulus $A$ that are equal to the identity in some neighbourhood of
  the boundary; 
\item $\Symp(A,\partial A,\ostd)\subset \Diff(A,\partial A)$ the subgroup of
  symplectomorphisms of $(A,\ostd)$;
\item $\Symp_0(A,\partial A,\ostd)$ the identity component of
  $Symp(A,\partial A,\ostd)$.
\end{itemize}

{\bf Diffeomorphisms. }
All results in this appendix are based on the following fundamental
theorem of Smale. 

\begin{theorem}[Smale~\cite{Sma}]\label{thm:Smale}
The group $\Diff(D,\partial D)$ is contractible.
\end{theorem}

With a nondecreasing cutoff function $\rho:\R\to[0,1]$ which equals
$0$ near $(-\infty,a]$ and $1$ near $[b,\infty)$, we define the {\em
    Dehn twist}
$$
   \phi^D \colon A \to A,\qquad 
   re^{i\theta}\mapsto re^{i(\theta+2\pi \rho(r))}.
$$

\begin{cor}\label{TDiff}
All homotopy groups $\pi_i\Diff(A,\partial A)$ vanish except for the group
$\pi_0\Diff(A,\partial A)=\mathbb{Z}$, which is generated by the Dehn
twist $\phi^D$. 
\end{cor}

\begin{proof}
Restriction of elements in $\Diff(D,\p D)$ to the smaller disk
$D_a\subset D$ of radius $a$ yields a Serre fibration
$$
   \Diff(A,\p A)\to \Diff(D,\p D)\to \Diff^+(D_a),
$$
where $\Diff^+$ denotes the orientation preserving diffeomorphisms. 
In view of Smale's Theorem~\ref{thm:Smale}, the long exact sequence of
this fibration yields isomorphisms $\pi_i\Diff^+(D_a)\cong
\pi_{i-1}\Diff(A,\p A)$ for all $i\geq 1$. Again by
Theorem~\ref{thm:Smale}, the long exact sequence of the pair
$(\Diff^+(D_a),\Diff^+(\p D_a))$ yields isomorphisms $\pi_i\Diff^+(\p
D_a)\cong\pi_i\Diff^+(D_a)$ for all $i$. Since $\pi_i\Diff^+(\p
D_a)\cong \pi_i\Diff^+(S^1)$ equals $\Z$ for $i=1$ and $0$ otherwise,
this proves the corollary. 
\end{proof}

For the following slight refinement of Corollary~\ref{TDiff}, let
$E\subset A$ be a circle $\{c\}\times S^1$ for some $c\in(a,b)$.  

\begin{cor}\label{cor:Diff}
Every smooth loop $(\phi_t)_{t\in[0,1]}$ in $\Diff(A,\partial A)$ with
$\phi_0=\phi_1=\id$ and $\phi_t(E)=E$ for all $t$ can be contracted by
a smooth family $\phi_t^s\in\Diff(A,\p A)$, $s,t\in[0,1]$, satisfying
$\phi_t^0=\phi_0^s=\phi_1^s=\id$, $\phi_t^1=\phi_t$ and $\phi_t^s(E)=E$
for all $s,t$. 
\end{cor}

\begin{proof}
For a point $e\in S^1$ the family of arcs $\phi_t([a,c]\times\{e\})$
starts and ends at $t=0,1$ with the arc $[a,c]\times\{e\}$. This shows
that the loop $t\mapsto\phi_t(c,e)$ in $E$ is contractible, hence so
is the loop $\phi_t|_E$ in $\Diff(E)$. Thus we can find a family  
$\phi_t^s\in\Diff(A,\p A)$, $s,t\in[1/2,1]$, satisfying
$\phi_0^s=\phi_1^s=\id$, $\phi_t^1=\phi_t$, $\phi_t^s(E)=E$
for all $s,t$, and $\phi_t^{1/2}|_E=\id$ for all $t$. Now apply
Corollary~\ref{TDiff} to contract the loops
$\phi_t^{1/2}|_{[a,c]\times S^1}$ in $\Diff([a,c]\times
S^1,\p[a,c]\times S^1)$ and $\phi_t^{1/2}|_{[c,b]\times S^1}$ in 
$\Diff([c,b]\times S^1,\p[c,b]\times S^1)$.
\end{proof}

{\bf Symplectomorphisms. }
The following result is an immediate consequence of
Corollary~\ref{TDiff} and Moser's theorem. 

\begin{proposition}\label{PHE}
The groups $\Diff(A,\partial A)$ and $\Symp(A,\partial A,\ostd)$ are weakly
homotopy equivalent. In particular, $\Symp_0(A,\partial A,\ostd)$ is
contractible. 
\end{proposition}

\begin{remark}\label{rem:Dehn}
The Dehn twist $\phi^D$ defines an element in $\Symp(A,\partial A,\ostd)$,
which according to Proposition~\ref{PHE} generates
$\pi_0\Symp(A,\partial A,\ostd)$. 
\end{remark}

Finally, we need the following refinement of
Proposition~\ref{PHE}. Again, let $E\subset A$ be a circle $\{c\}\times S^1$
for some $c\in(a,b)$. 

\begin{lemma}\label{L1}
Each $\phi \in \Symp_0(A,\partial A,\ostd)$ with $\phi(E)=E$ can be
connected to the identity by a smooth path $\phi_t \in
\Symp_0(A,\partial A,\ostd)$ satisfying $\phi_t(E)=E$ for all
$t\in[0,1]$. 
\end{lemma}

\begin{proof}
After applying Moser's theorem and changing the values of $a,b,c$ (viewing
$A$ as a cylinder), we may assume that $\sigma_\std=dr\wedge d\theta$
and $c=0$. We connect $\phi$ to the identity in 4 steps.  

{\em Step 1. }The restriction $f(\theta):=\phi(0,\theta)$ of $\phi$ to
$E$ defines an element in the group $\Diff^+(S^1)$ of orientation preserving
diffeomorphisms of the circle. Since this group is path connected,
there exists a smooth family $f_t\in\Diff^+(S^1)$ with $f_0=\id$ and
$f_1=f$. This family if generated by the time-dependent vector field
$\xi_t$ on the circle defined by $\xi_t(f_t(\theta)):=\dot
f_t(\theta)$. Let $H_t:A\to\R$ be a smooth family of functions
satisfying $H_t(r,\theta) = -r\xi_t(\theta)$ near $E$ and $H_t=0$
near $\p A$. A short computation shows that the Hamiltonian vector
field of $H_t$ agrees with $\xi_t$ on $E$. It follows that the
Hamiltonian flow $\psi_t$ of $H_t$ satisfies $\psi_t|_E=f_t$, in
particular $\psi_1|_E=f_1=\phi|_E$. Thus $\phi_t:=\phi\circ\psi_t^{-1}$ 
is a smooth path in $\Symp_0(A,\partial A,\ostd)$ with $\phi_t(E)=E$ 
connecting $\phi$ to $\phi_1$ satisfying $\phi_1|_E=\id$. 
After renaming $\phi_1$ back to $\phi$, we may hence assume that
$\phi|_E=\id$. 

{\em Step 2. }
Let us write
$\phi(r,\theta)=\bigl(P(r,\theta),Q(r,\theta)\bigr)\in\R\times
S^1$. Since $\phi|_E=\id$ and $\phi$ is symplectic, the functions
$P,Q$ satisfy 
$$
   P(0,\theta)=0,\qquad Q(0,\theta)=\theta,\qquad \frac{\p P}{\p
     r}(0,\theta)=1.  
$$
For $s\in(0,1]$ consider the dilations $\tau_s(r,\theta):=(sr,\theta)$
on $A$. Since $\tau_s^*(dr\wedge d\theta)=s\,dr\wedge d\theta$, the
maps $\psi_s:=\tau_s^{-1}\circ\phi\circ\tau_s$ are symplectic. Since 
$$
   \psi_s(r,\theta) = \Bigl(\frac{1}{s}P(sr,\theta),Q(sr,\theta)\Bigr)
   \xrightarrow[s\to 0]{}\Bigl(r\frac{\p P}{\p
     r}(0,\theta),Q(0,\theta)\Bigr) 
   = (r,\theta),
$$ 
the family $\psi_s$ extends smoothly to $s=0$ by the identity 
(this is a fibered version of the Alexander trick). 
It follows that for a sufficiently small $\eps>0$ we have a smooth
family of symplectic embeddings $\psi_s:A_\eps:=[-\eps,\eps]\times
S^1\into A$, $s\in[0,1]$, with $\psi_s(E)=E$, $\psi_0=\id$, and
$\psi_1=\phi$. We extend this family to smooth diffeomorphisms
$\wt\psi_s:A\to A$ with $\wt\psi_s=\id$ near $\p A$ and
$\wt\psi_1=\phi$. Since $\wt\psi_s$ preserves the annuli 
$A^-:=[a,0]\times S^1$ and $A^+:=[0,b]\times S^1$, it satisfies
$\int_{A^\pm}\wt\psi_s^*\ostd=\int_{A^\pm}\ostd$ for all
$s\in[0,1]$. By Banyaga's Theorem~\ref{thm:Banyaga}, applied to the
isotopy $t\mapsto\phi^{-1}\circ\wt\psi_{1-t}$ and the set
$X:=[a,a+\eps]\times S^1\cup [b-\eps,b]\times S^1\cup
A_\eps$ for some possibly smaller $\eps>0$, there exists a symplectic
isotopy $\phi_s:A\to A$, $s\in[0,1]$, with $\phi_1=\phi$ and
$\phi_s|_X=\wt\psi_s|_X$. In particular, $\phi_s\in\Symp_0(A,\p
A,\ostd)$ preserves $E$ and $\phi_0|_{A_\eps}=\id$. 
After renaming $\phi_0$ back to $\phi$, we may hence assume that
$\phi=\id$ on an annulus $A_\eps$ around $E$. 

{\em Step 3. }
Since $\phi|_{A_\eps}=\id$, it restricts to maps
$\phi|_{A^\pm}\in\Symp(A^\pm,\p A^\pm,\ostd)$.  
By Proposition~\ref{PHE} and Remark~\ref{rem:Dehn}, $\phi|_{A^\pm}$
can be connected in $\Symp(A^\pm,\p A^\pm,\ostd)$ to a multiple
$(\phi_\pm^D)^{k_\pm}$ of the Dehn twist on $A^\pm$. Since $\phi$
belongs to the identity component $\Symp_0(A,\p A,\ostd)$, it follows
that $k_+=-k_-$. Hence we can simultaneously unwind the Dehn twists to
connect the map $\psi$ which equals $(\phi_\pm^D)^{k_\pm}$ on $A^\pm$
to the identity by a path $\psi_t$ in $\Symp_0(A,\p A,\ostd)$ fixing
$E$ (but not restricting to the identity on $E$). Thus
$\phi_t:=\phi\circ\psi_t^{-1}$ is a path in $\Symp_0(A,\partial
A,\ostd)$ with $\phi_t(E)=E$ connecting $\phi$ to $\phi_1$ such that
$\phi_1|_{A^\pm}$ belongs to the identity component $\Symp_0(A^\pm,\p
A^\pm,\ostd)$. Again, we rename $\phi_1$ back to $\phi$. 

{\em Step 4. }Finally, we apply Proposition~\ref{PHE} on $A^\pm$ to
connect $\phi|_{A^\pm}$ to the identity by a path $\phi_t^\pm$ in
$\Symp_0(A^\pm,\p A^\pm)$. The maps $\phi_t^\pm$ fit together to 
a path $\phi_t\in\Symp_0(A,\p A,\ostd)$ fixing $E$ that connects
$\phi$ to the identity.  
\end{proof}

\end{document}